\documentclass[11pt]{article} 
\usepackage[top=1in, bottom=1in, left=1in, right=1in]{geometry} 
\usepackage{adjustbox}

\input{articleStyle.sty}
\graphicspath{ {./images/} } 

\title{A Nonconvex Framework for Joint Sparse Sample Weighting and Parameter Estimation}

\author[1]{Hassan Nojavan}
\author[2]{Chul Moon}
\author[1]{Miju Ahn \thanks{Emails: {hnojavan@smu.edu}, {chulm@smu.edu} and {mijua@smu.edu}}}

\affil[1]{Department of Operations Research and Engineering Management, Southern Methodist University}
\affil[2]{Department of Statistics and Data Science, Southern Methodist University}

\date{First submission: October 12, 2026}

\begin{document}

\maketitle
\begin{abstract}
We propose a joint sparse sample weighting and parameter estimation (SWAPE) framework that jointly optimizes model parameters and sample weights to train a model while simultaneously assigning an individual weight to each training data point. Our framework accommodates several widely used loss functions for the model parameters, including convex, weakly convex, and difference-of-convex loss, and weakly convex sparsity functions to control the complexity of the sample weights. To solve the resulting bivariate problem, we propose an alternating minimization method in which the subproblems for updating the model parameters and weights are each approximated using majorize-minimize techniques. We provide several approximation methods for the loss functions considered in this paper, and show the overall algorithm's convergence. Numerical experiments on various learning tasks and data anomaly settings demonstrate the effectiveness of the proposed framework in minimizing the impact of undesirable data points compared to benchmark methods. 
\end{abstract}


\medskip

\section{Introduction}

Judicious sample selection is an important task in machine learning. It is well accepted that the quality of the training data directly affects the predictive model; however, many standard algorithms assume a clean dataset, whereas in practice observations often contain noise, outliers, and other anomalies. Therefore, training predictive models using strategically selected samples has emerged as a critical methodology for building robust models. 

Recent results have shown that selective use of data points can boost model performance and mitigate the unwanted effects of noisy data, supporting the development of reliable and efficient tools. A conventional data selection method is sampling, e.g., simple sampling, which randomly draws data points one at a time until a desired sample size, or a desired ratio among predetermined data groups is obtained, e.g., under-, over-, and stratified sampling. Alternatively, algorithm-based approaches iteratively assign weights to data points \cite{friedman2001greedy, friedman2000additive, ren2018learning}, measure data values \cite{pmlr-v97-ghorbani19c}, or select a subset of data points so that the model emphasizes accuracy on the most significant samples \cite{rousseeuw2006computing}. Formulated as a mathematical program, sample selection has also been modeled as an optimization problem, for example, by introducing a binary variable \cite{NEURIPS2021_07563a3f}, to indicate whether each data point should be included in the training criteria.

This paper proposes an optimization framework for sample selection that exploits mature variable selection methods \cite{fan2001variable, tibshirani96lasso, zhang2010nearly}. Instead of driving the coefficients of insignificant variables to zero as in variable selection, our formulation employs sparsity regularizers in the training problem to control the complexity of the training data used by the model. We assign a weight variable to the loss incurred by each sample and solve for the weight, 
allowing the model to choose which data points to use during training. The resulting problem can be viewed as a widely studied weighted empirical risk minimization problem, except that our method simultaneously optimizes both the data weights and the model parameters. The introduction of weights as variables, in addition to the model coefficients, makes solving the resulting bivariate optimization problem significantly more challenging. For instance, our formulation can be nonconvex even with a convex loss term, limiting the availability of efficient numerical tools in convex programming. A major goal of our study is to address the computational challenges posed by the mathematical properties given by the choice of loss functions and to provide numerical results demonstrating the benefits of sample selection when dealing with potentially noisy and contaminated data. 

Our algorithm is built upon existing numerical methods that alternately solve for one variable at a time. Such an iterative scheme for solving multi-block problems, known as the Gauss-Seidel method \cite{bertsekas1989parallel}, has received considerable attention in the literature, especially with recent innovations in large-scale, data-driven methods. The convergence of the algorithm has been shown in early works in convex and continuously differentiable settings \cite{BeckTetruashvili2013, Warga1963, wright2015coordinate}
and later extended to nonconvex and nonsmooth problems \cite{Bolte2013, razaviyayn2013unified, razaviyayn2014parallel, Tseng2001}. A key component to tackling the computational challenge in the latter problems is to approximate each subproblem, e.g., by linearization or convex approximation, so that each iteration can be solved efficiently using readily available tools. In this work, we adopt the blockwise minimization methods to alternately minimize the weights of the data points and the model parameters. This work emphasizes nonconvex loss functions, including a difference-of-convex (DC) loss, that generalizes many existing loss functions used in practice, e.g., convex and weakly convex loss functions. We approximate each subproblem using majorization techniques \cite{hunter2004tutorial} and analyze the algorithm's convergence. Our results extend the existing theory of the alternating minimization scheme for the two-block problem, in which each block, while fixing the other, can be written as a DC function.

Our work is motivated by the growing imperative in modern statistics and machine learning to learn with better utilization of training data. In robust regression, \cite{gao2018penalized, jiang2021outlier} demonstrate that simultaneously training model parameters and optimizing data weights can significantly strengthen performance. However, many existing approaches are tailored to specific formulations, such as a convex loss paired with specialized sparsity functions, and often make restrictive assumptions about the underlying properties of the data. To address these limitations, we propose a generalized framework that accommodates a broad class of loss and sparsity functions and an algorithm that solves the resulting formulation efficiently, without requiring a priori knowledge of the data's underlying structure. We summarize the main contributions of this work as follows:

\begin{enumerate}[noitemsep]
\item We propose a bivariate formulation that simultaneously minimizes the weight of data points and model parameters. Our method aims to promote selective utilization of training data points by applying sparsity regularizers to the weights, thereby performing data selection while training the model;  
\item We employ an alternating minimization scheme to solve the proposed nonconvex problem, and discuss the majorization techniques used for each subproblem for DC loss and functions with a Lipschitz-continuous gradient. Our convergence analysis extends existing theory of the two-block alternating minimization scheme for the special case studied in this paper where the objective is a bi-DC function;
\item We conduct extensive computational experiments to demonstrate the efficacy of our method compared to suitable benchmark models. For this purpose, we test the proposed method in linear regression, binary classification via neural networks, and recommender systems under various data-anomaly settings: ($i$) when data points follow different underlying distributions, ($ii$) when the labels in the classification problems are contaminated, and ($iii$) when the data contains random noise and adversarial attacks. 
\end{enumerate}

This paper is organized as follows. In Section~\ref{sec:framework}, we introduce our framework, including the nonconvex loss and penalty functions used in this work, and define the concept of stationarity for subsequent analysis. Section~\ref{sec:algorithm} presents our algorithm, subproblem approximations, and the convergence analysis, followed by three numerical study experiments in Section~\ref{sec:numerical_experiment}. This document presents theoretical results and omits proofs; the proofs are included in the Appendices. We close this section by defining notations used throughout the paper.  

\textbf{Notations}: Let $f(\x,\y)$ be a bivariate function in the variables $\x$ and $\y$. We define $f_{\tilde{\y}} (\x) \triangleq f(\x, \tilde{\y})$ a function of $\x$ while fixing $\y = \tilde{\y}$. The function $f_{\tilde{\x}}(\y)$ is similarly defined.
The directional derivative of a function $f:\mathbb{R}^n \rightarrow \mathbb{R}$ at a point $\x \in \mathbb{R}^n$ in the direction $\h d \in \mathbb{R}^n$ is defined as  
\begin{align} \label{def:d_deriv}
     f'(\x; \h d) \triangleq \lim \limits_{\tau \downarrow 0} \displaystyle{\frac{f( \, \x + \tau \h d \,) - f(\x)}{\tau}}.
\end{align}
For a vector \(\x \in \mathbb{R}^n_+\), we define
\(
    \x^{1/2} \triangleq
    \left(\sqrt{x_1},\ldots,\sqrt{x_n}\right)^\mathsf{T}.
\)
For a nonnegative matrix \(X\), the notation \(X^{1/2}\) is defined similarly, entrywise. For vectors \(\h u,\h v \in \mathbb{R}^n\), \(\h u \odot \h v\) denotes the Hadamard product, \(\h u \odot \h v \triangleq (u_1v_1,\ldots,u_nv_n)^\mathsf{T}.\) Likewise, the Hadamard product of two matrices of the same size is their elementwise product.

\section{Framework} \label{sec:framework}
Consider a supervised learning framework in which each data point is assigned an unknown weight that captures its significance to the model. Let $\h \xi^i \in \mathbb{R}^{d+1}$ be the $i$-th observation, containing $d$ input features and a response. This results in the training set $\{ \h \xi^i \}_{i=1}^n$. Denoting the model coefficients as $\x \in \mathbb{R}^d$ and the weight vector as $\w \in \mathbb{R}^n$, we propose a bivariate optimization problem where both $\x$ and $\w$ are computed simultaneously during the training process. We refer to the proposed method as the joint sparse sample weighting and parameter estimation (SWAPE), which solves:
\begin{align} \label{prob:mainprob}
    \minimize \limits_{\x \in \mathcal{X}, \w \in \mathcal{W}} \
    &F(\x, \w) = \underbrace{ \sum \limits_{i=1}^n w_i \, \ell(\x; \, \h \xi^i) }_{\triangleq \mathcal{L}(\x, \w)} + \lambda \, R(\x) + \gamma \, \underbrace{\sum \limits_{i=1}^n \rho(w_i)}_{\triangleq P(\w)}, 
\end{align}
where $\mathcal{X}$ is a compact convex set and $\mathcal{W}$ is a box constraint set defined as $\mathcal{W} \triangleq \{ \, \w \, ~\vert~ 0 \leq w_i \leq u_i \text{ for some } u_i > 0 \quad \forall \ i \in \{1, \dots, n\} \, \}$. 

We assume the model's loss measured by $\ell:\mathbb{R}^d \rightarrow \mathbb{R}_{+}$ is nonnegative. Therefore, we restrict the weights to take nonnegative values in the weighted loss, denoted as $\mathcal{L}:\mathcal{X} \times \mathcal{W} \rightarrow \mathbb{R}_{+}$. Additionally, a penalty function, $P:\mathbb{R}^n \rightarrow \mathbb{R}$, captures the complexity of the weights of the data points used in the model. We employ and modify existing sparsity functions for this purpose. In \eqref{prob:mainprob}, a convex and lower-bounded regularizer, $R: \mathbb{R}^d \rightarrow \mathbb{R}$, is included to accommodate existing regularized methods. The scalars $\lambda, \, \gamma > 0$ determine the weight of the regularizer and penalty, respectively. 
The computational challenges of solving \eqref{prob:mainprob} is associated with the choices of $\ell$ and $\rho$. In Section~\ref{sec:loss_penalty}, we discuss more details and list the requirements made on the loss and penalty functions. 

\subsection{Assumptions on the loss and penalty functions} \label{sec:loss_penalty}
This work focuses on two classes of loss functions for training the model coefficients $\x$. For a fixed arbitrary $\w$, we assume that loss function $\mathcal{L}_{\w}(\x)$ satisfies one of the following conditions:
\begin{enumerate}[label=(L\arabic*)] 
\item \hspace{-0.45pc} \label{assume:L_dc} The function $\mathcal{L}_{\w}(\x)$ is a difference-of-convex (DC) function, i.e., there exist two convex functions $\mathcal{L}_{\w}^{\rm cvx} (\x)$ and $\mathcal{L}_{\w}^{\rm cve} (\x)$ such that $\mathcal{L}_{\w}(\x) = \mathcal{L}_{\w}^{\rm cvx} (\x)-\mathcal{L}_{\w}^{\rm cve} (\x) \text{ for any } \x \in \mathcal{X}$;

\item \hspace{-0.45pc} \label{assume:L_lipgrad} 
There exists $\lip > 0$ such that 
\begin{align} \label{assume:Lip_ineq}
    \mathcal{L}_{\w}(\x) - \mathcal{L}_{\w}(\h y) - \la \nabla \mathcal{L}_{\w}(\h y), \x - \h y \ra \leq \frac{\lip}{2} \| \x - \h y \|_2^2 \quad \forall \ \x, \h y \in \mathcal{X}.
\end{align}
\end{enumerate}
The first class of loss functions, shown in \ref{assume:L_dc}, includes convex loss functions, such as generalized linear models \cite{James2021}, and nonconvex loss functions with a DC representation, e.g., a shallow feedforward neural network \cite{awasthi2024dc}. 

The second class of loss functions, in \ref{assume:L_lipgrad}, includes functions with a Lipschitz-continuous gradient.
The condition \eqref{assume:Lip_ineq} is also closely related to weakly convex functions; a function $f$ is called weakly convex if there exists a weak convexity modulus $\sigma > 0$ such that $f (\x) + \frac{\sigma}{2}\| \x \|_2^2$ is convex. With some algebraic work, one can show \eqref{assume:Lip_ineq} is equivalent to $-\mathcal{L}(\x)$ being a weakly convex function.
Additionally, it is not difficult to see that weakly convex functions form a subclass of DC functions since one can write a weakly convex function $f$ as $f (\x) =  f (\x) + \frac{\sigma}{2}\| \x \|_2^2 - \frac{\sigma}{2}\| \x \|_2^2$, a difference of two convex functions $f (\x) + \frac{\sigma}{2}\| \x \|_2^2$ and $\frac{\sigma}{2}\| \x \|_2^2$. However, the converse relationship does not hold in general. To complete the relationship between DC loss and weakly convex functions, Proposition~\ref{prop:DcWc} identifies a condition for a DC loss $\mathcal{L}_{\w} (\x)$ satisfying \ref{assume:L_dc} to achieve weakly convexity.
\begin{proposition} \label{prop:DcWc}
    Let \ref{assume:L_dc} hold for $\mathcal{L}_{\w}(\x)$. If there exists $\lipLcve > 0$ such that
        \begin{align} \label{Lcve_Lipgrad}
        \mathcal{L}_{\w}^{\rm cve} (\h y) - \mathcal{L}_{\w}^{\rm cve} (\x) - \mathcal{L}_{\w}^{\rm cve} \, ' (\x; \h y - \x) \leq \frac{\lipLcve}{2} \| \x - \h y \|_2^2 \quad \forall \x, \, \h y \in \mathcal{X},
        \end{align}
    then $\mathcal{L}_{\w}(\x)$ is weakly convex with modulus $\lipLcve$.
\end{proposition}

We next discuss mathematical properties of the univariate sparsity regularizer $\rho$.
It is well known that most existing sparsity functions, e.g., the $\ell_1$-norm used in LASSO \cite{tibshirani96lasso}, aim to promote sparsity in the model coefficients by setting the values to zero.
However, directly substituting such a function into $\rho$ in \eqref{prob:mainprob} is not meaningful. For instance, with $R(\x) = 0$, we have a trivial solution for the problem by setting $\w = 0$. Therefore, we consider a modified version of the sparsity functions that promote the weight variable to take the value $1$ rather than $0$. With this modification, the sparsity function $\rho$ gives the model the flexibility to use (or omit) the $i$-th data point during training at a penalty. Below, we present assumptions made on the sparsity function $\rho$ in the scalar variable $t \in \mathbb{R}$: 
\begin{enumerate}[label=(P\arabic*)] 
\setcounter{enumi}{2}

\item \hspace{-0.45pc} \label{assume:p_wc} The function $\rho:\mathbb{R} \rightarrow \mathbb{R}$ is a weakly convex function, i.e., there exists $\sigma_{\rho} >0$ such that $\rho(t) + \frac{\sigma_{\rho}}{2} t^2$ is convex;

\item \hspace{-0.45pc} \label{assume:p_foldedconcave} The function $\rho(t)$ is symmetric about $t=1$, satisfies $\rho(t) = 0$ when $t=1$, and is concave and non-decreasing on the interval $[1,\infty)$.
\end{enumerate}
The assumptions \ref{assume:p_wc} and \ref{assume:p_foldedconcave} include existing folded-concave sparsity functions, such as SCAD \cite{fan2001variable}, MCP \cite{zhang2010nearly}, and transformed $\ell_1$ penalty \cite{Nikolova00}, with a mild modification of shifting the function horizontally by a constant. Figure~\ref{fig:penalty-plots-1} illustrates variants of LASSO and MCP that meet our requirements. Both functions have been horizontally shifted by one unit such that they are symmetric at $t=1$.

\subsection{Concept of stationarity}
The problem \eqref{prob:mainprob} is nonconvex and nondifferentiable; therefore, one should carefully select the type of stationary solutions to analyze the problem. For this purpose, we first show that \eqref{prob:mainprob} is directionally differentiable under the assumptions made on the loss and the sparsity regularizer. Lemma~\ref{lem:d_deriv} below shows this result.
\begin{lemma} \label{lem:d_deriv}
    Let \ref{assume:L_dc} or \ref{assume:L_lipgrad} hold. Under \ref{assume:p_wc}, the function $F (\h x, \h w)$ in \eqref{prob:mainprob} is directionally differentiable, and its directional derivative at the point $\h z \triangleq (\h x, \h w)$ in the direction $\h d \triangleq (\h d^{\h x}, \h d^{\h w}) \in \mathbb{R}^{d+n}$ is written as 
    \begin{align} \label{lem:d_deriv_split}
        F'(\h z; \h d) = F_{\w}'(\h x; \h d^{\h x}) + F_{\x}'(\h w; \h d^{\h w}).
    \end{align}
\end{lemma}

\begin{figure}[t] \centering \includegraphics[width=0.7\linewidth]{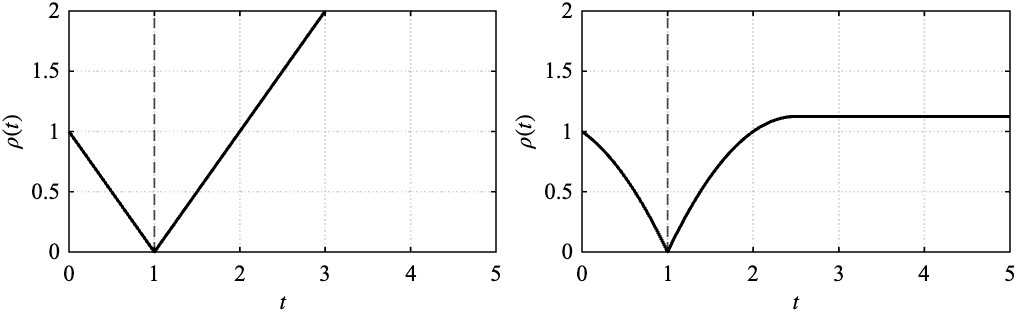} \caption{ Penalty functions: LASSO ($\lambda = 1$) on the left, and MCP ($\eta = 1$, $\theta = 1.5$) on the right.} \label{fig:penalty-plots-1} \end{figure}

Note that a function that is not differentiable in the ordinary sense may be directionally differentiable. For example, it has been shown that the directional derivative of a convex function is well-defined at any point \cite{NLP_Bazarra_et_al}.
Lemma~\ref{lem:d_deriv} shows that, under the stated assumptions, the problem \eqref{prob:mainprob} has a directionally differentiable objective subject to a convex constraint set. For such a structured problem, the directional stationary solution, used in nonconvex and nonsmooth optimization \cite{Cui2021}, is a suitable solution type for analyzing the problem. We formally define a directional stationary solution below. 
\begin{definition}
Consider a closed convex set $\mathcal{X}$. A point $\h x^* \in \mathcal{X}$ is a directional stationary solution of the problem $\min \limits_{\h x \in \mathcal{X}} f(\h x)$ if  $f'(\x^*; \x - \x^*) \geq 0$ for all $\x \in \mathcal{X}$. 
\end{definition}
We next propose an algorithm to solve \eqref{prob:mainprob} and then show the convergence analysis using the directional stationary solution.

\section{Algorithm} \label{sec:algorithm}
Instead of a direct minimization of $F(\x, \w)$ in \eqref{prob:mainprob}, we iteratively solve for $\x$ and $\w$ while fixing the other variable. Employing the majorize-minimization (MM) framework \cite{hunter2004tutorial,lange2016mm} for each subproblem, we propose iteratively minimizing a surrogate function, which is an upper-bound estimation given by the previous iterate. Our approach is a special case of the block successive upper-bound minimization algorithm \cite{razaviyayn2013unified} with a two-block structure.

\subsection{Majorize-minimization framework} \label{sec:MM}
To introduce our framework, let us consider problem \eqref{prob:mainprob} in the variable $\x$ while fixing the variable $\w = \w^t$. We define a resulting function only involving terms associated with $\x$: 
\begin{align} \label{MM:F}
    \widehat{F}_{\w^t}(\x) \triangleq  \mathcal{L}_{\w^t}(\x) + \lambda R(\x).
\end{align}
For any given $\x^t$, the surrogate function denoted by $\widehat{G}_{\w^t}:\mathbb{R}^d \rightarrow \mathbb{R}$, shall satisfy the following conditions:
\begin{align} \label{MM:cond}
    \widehat{F}_{\w^t}(\x^t) = \widehat{G}_{\w^t}(\x^t; \x^t) \quad \text{and} \quad \widehat{F}_{\w^t}(\x) \leq \widehat{G}_{\w^t}(\x; \x^t) \ \forall \x \in \mathcal{X}.
\end{align}
Likewise, while fixing variable $\x = \x^{t+1}$, define 
\begin{align} \label{MM:F_W}
    \widehat{F}_{\x^{t+1}}(\w) \triangleq \mathcal{L}_{\x^{t+1}}(\w) + \gamma P(\w)
\end{align}
and a function $\widehat{G}_{\x^{t+1}}: \mathbb{R}^n \rightarrow \mathbb{R}$ to satisfy
\begin{align} \label{MM:cond_W}
    \widehat{F}_{\x^{t+1}}(\w^t) = \widehat{G}_{\x^{t+1}}(\w^t; \w^t) \quad \text{and} \quad \widehat{F}_{\x^{t+1}}(\w) \leq \widehat{G}_{\x^{t+1}}(\w; \w^t) \ \forall \w \in \mathcal{W},
\end{align}
for any $\w^t$.
In \eqref{MM:cond}, the function $\widehat{G}_{\w^t}(\x; \x^t)$ serves as an upper bound of $\widehat{F}_{\w^t}(\x)$ and, likewise, $\widehat{G}_{\x^{t+1}}(\w; \w^t)$ for $\widehat{F}_{\x^{t+1}}(\w)$ in \eqref{MM:cond_W}. Our method sequentially minimizes the discussed upper approximation problems of \eqref{prob:mainprob} in each variable to obtain the next iterate. The overall framework is presented below:
\begin{equation}\label{alg:subproblems}
\left\{\begin{array}{l}
\x^{t+1} \in \argmin \limits_{\x \in \mathcal{X}} \widehat{G}_{\w^t}(\x; \x^t), \quad \text{such that } \widehat{G}_{\w^t}(\x; \x^t) \text{ satisfies } \eqref{MM:cond}; \\
\w^{t+1} \in \argmin \limits_{\w \in \mathcal{W}} \widehat{G}_{\x^{t+1}}(\w; \w^t) \quad \text{such that } \widehat{G}_{\x^{t+1}}(\w; \w^t) \text{ satisfies } \eqref{MM:cond_W}.
\end{array}\right.
\end{equation}

\textbf{The $\mathbf{\x}$-problem.}
Next, we present $\widehat{G}_{\w^t}(\x; \x^t)$ for different cases of loss functions considered in this paper. 
First, for the case of DC loss, i.e., $\mathcal{L}_{\w^t} (\x)$ satisfies \ref{assume:L_dc}, we apply the first-order approximation for the concave component and obtain
\begin{align} \label{MM:dc}
    \widehat{G}_{\w^t}(\x; \x^t) \triangleq \mathcal{L}_{\w^t}^{\rm cvx} (\x) - \mathcal{L}_{\w^t}^{\rm cve} (\x^t) - \la \zeta_{t}, \x - \x^t \ra + \lambda \, R(\x) + \frac{c_{\x}}{2} \| \x - \x^t \|_2^2,
\end{align}
where $\zeta_{t} \in \partial \mathcal{L}_{\w^t}^{\rm cve} (\x^t)$ and $c_{\x}$ is a nonnegative scalar. Such an estimation is widely used to minimize a DC function, e.g., DC algorithm in \cite{tao1997convex}. We formally state the relationship between $\widehat{F}_{\w^t}(\x)$ and $\widehat{G}_{\w^t}(\x; \x^t)$ in \eqref{MM:dc} in Lemma~\ref{lemma:MM_DC}. 
\begin{lemma} \label{lemma:MM_DC}
    Let \ref{assume:L_dc} hold. The functions $\widehat{F}_{\w^t}(\x)$ and $\widehat{G}_{\w^t}(\x; \x^t)$, defined in \eqref{MM:F} and \eqref{MM:dc}, satisfy the condition \eqref{MM:cond} for any $\x^t \in \mathcal{X}$.
\end{lemma}

For the case of the loss function  $\mathcal{L}_{\w^t} (\x)$ satisfies \ref{assume:L_lipgrad}, we construct
\begin{align} \label{MM:lipgrad}
    \widehat{G}_{\w^t}(\x; \x^t) \triangleq \mathcal{L}_{\w^t} (\x^t) + \la \nabla \mathcal{L}_{\w^t} (\x^t), \x - \x^t \ra + \lambda \, R(\x) + \frac{c_{\x}}{2} \| \x - \x^t \|_2^2,
\end{align}
for some $c_{\x} > 0$, following the approximation techniques used in the literature \cite{Bolte2013,razaviyayn2014parallel}.
The relationship between \eqref{MM:F} and \eqref{MM:lipgrad} is shown in Lemma~\ref{lemma:MM_lipgrad}.
\begin{lemma} \label{lemma:MM_lipgrad}
    Let \ref{assume:L_lipgrad} hold. If $c_{\x} \geq \lip$, the functions $\widehat{F}_{\w^t}(\x)$ and $\widehat{G}_{\w^t}(\x; \x^t)$, defined in \eqref{MM:F} and \eqref{MM:lipgrad}, satisfy the condition \eqref{MM:cond}.
\end{lemma}

Lastly, if $\mathcal{L}_{\w^t} (\x)$ in \eqref{prob:mainprob} satisfies assumption \ref{assume:L_dc} with \eqref{Lcve_Lipgrad}, i.e., is weakly convex, we construct
\begin{align} \label{MM:wc}
    \widehat{G}_{\w^t} (\x; \x^t) = \widehat{F}_{\w^t}(\x) + \frac{c_{\x}}{2} \| \x - \x^t \|_2^2,
\end{align}
for some $c_{\x} \geq 0$. We observe that the function $ \widehat{G}_{\w^t} (\x; \x^t)$ in \eqref{MM:wc} can be convexified by choosing $c_{\x}$ greater than the weakly convexity modulus of $\mathcal{L}_{\w^t} (\x)$. Lemma~\ref{lemma:MM_wc} follows directly from Proposition~\ref{prop:DcWc} and states this property. 
\begin{lemma} \label{lemma:MM_wc}
    The functions $\widehat{F}_{\w^t}(\x)$ and $\widehat{G}_{\w^t}(\x; \x^t)$, defined in \eqref{MM:F} and \eqref{MM:wc} respectively, satisfy the condition \eqref{MM:cond}. Moreover, if $\mathcal{L}_{\w^t} (\x)$ satisfies \ref{assume:L_dc} such that inequality \eqref{Lcve_Lipgrad} holds and $c_{\x} \geq \lipLcve$, then $\widehat{G}_{\w^t} (\x; \x^t)$ in \eqref{MM:wc} is convex.
\end{lemma}

\textbf{The $\mathbf{\w}$-problem.} 
Recall our method in \eqref{alg:subproblems} alternately updates $\x$ and $\w$ by minimizing an approximation in each subproblem. Next, we present the surrogate function for updating the variable $\w$: 
\begin{align} \label{MM:wc_W}
    \widehat{G}_{\x^{t+1}}(\w; \w^t) \triangleq \widehat{F}_{\x^{t+1}}(\w) + \frac{c_{\w}}{2} \| \w - \w^t \|_2^2,
\end{align}
for some $c_{\w} \geq 0$. By adding a quadratic term, $\widehat{G}_{\x^{t+1}}(\w; \w^t)$ convexifies and serves as an upper bound for $\widehat{F}_{\x^{t+1}}(\w)$. This relationship between $\widehat{F}_{\x^{t+1}}(\w)$ and $\widehat{G}_{\x^{t+1}}(\w; \w^t)$ is stated in Lemma~\ref{lemma:MM_wc_w}.

\begin{lemma} \label{lemma:MM_wc_w}
    The functions $\widehat{F}_{\x^{t+1}}(\w)$ and $\widehat{G}_{\x^{t+1}}(\w; \w^t)$, defined in \eqref{MM:F_W} and \eqref{MM:wc_W}, satisfy condition \eqref{MM:cond_W}. Moreover, if \ref{assume:p_wc} holds with $c_{\w} \geq \gamma \sigma_{\rho}$, then $\widehat{G}_{\x^{t+1}}(\w; \w^t)$ is convex.
\end{lemma}

\subsection{Convergence analysis}
We show the convergence of our method in \eqref{alg:subproblems} for the $\x$- and $\w$-subproblem approximations given in Section~\ref{sec:MM}. 
Theorem~\ref{thm:conv} below summarizes our results.

\begin{theorem}\label{thm:conv}
Let assumptions \ref{assume:L_dc} - \ref{assume:L_lipgrad} and \ref{assume:p_wc}-\ref{assume:p_foldedconcave} hold. Consider the sequences $\{ (\x^t, \w^t) \}_{t=1}^\infty$ generated by \eqref{alg:subproblems} where one of the following holds for the $\x$-subproblem:
\begin{enumerate}[label=\textup{(\roman*)}]
    \item \label{thm:i}\ref{assume:L_dc} holds and $\widehat{G}_{\w^t}(\x; \x^t)$ is defined as \eqref{MM:dc} for some $c_{\x} > 0$; 
    \item \label{thm:ii} \ref{assume:L_lipgrad} holds and $\widehat{G}_{\w^t}(\x; \x^t)$ is defined as \eqref{MM:lipgrad} for some $c_{\x} > \lip > 0$;
    \item \label{thm:iii}  \ref{assume:L_dc} with \eqref{Lcve_Lipgrad} holds and 
    $\widehat{G}_{\w^t}(\x; \x^t)$ is defined as \eqref{MM:wc} for some $c_{\x} > 0$,
\end{enumerate}
    and $\widehat{G}_{\x^{t+1}} (\w; \w^t)$ is defined in \eqref{MM:wc_W} for some $c_{\w} > 0$. Then the sequence $\{\h (\h x^t,\h w^t)\}_{t=1}^\infty$ achieves the following properties:
\begin{enumerate}[label=\textup{(\alph*)}]
\item \label{thm:a}
      For every $t \geq 0$,
\begin{equation}
      F(\h x^t, \h w^t) -F(\h x^{t+1}, \h w^{t+1})
      \geq
      \frac{\bar{c}_{\h x}}{2}\,\bigl\|\h x^{t+1}-\h x^t\bigr\|_2^{2}
      +\frac{c_\h w}{2}\,\bigl\|\h w^{t+1}-\h w^t\bigr\|_2^{2},
      \label{eq:monodec}
\end{equation}
      for some nonnegative $\bar{c}_{\h x}$, and the sequence $\{ F(\h x^t, \h w^t ) \}_{t=1}^\infty$ converges;
\item \label{thm:b} We have
      \[
         \lim_{t\to\infty}\| (\h x^{t+1}, \h w^{t+1}) - (\h x^t, \h w^t) \|_2=0;
      \]
\item  The sequence $\{ (\h x^t,\h w^t) \}_{t=1}^\infty$ has a convergent subsequence denoted by $(\x^{t_j},\w^{t_j}) \rightarrow (\x^*, \w^*)$ as $t_j \rightarrow \infty$. For the case \ref{thm:iii}, the accumulation point $(\x^*, \w^*)$ is a directional stationary solution of problem \eqref{prob:mainprob}. 
\end{enumerate}
\end{theorem}

Note that for cases \ref{thm:i} and \ref{thm:ii},  $(\x^*, \w^*)$ is a directional stationary solution if $F_{\w^*}'(\h x^*; \h x - \h x^*) \geq 0 \ \forall \h x \in \mathcal{X}$ and $F_{\x^*}'(\h w^*; \h w - \h w^*) \geq 0 \ \forall \h w \in \mathcal{W}$. The proof for case \ref{thm:iii} does not require the assumptions \ref{assume:L_dc} and \eqref{Lcve_Lipgrad}, but these conditions ensure the $\x$-subproblem in \eqref{alg:subproblems} is convex.

\section{Numerical Experiments} \label{sec:numerical_experiment}
We conduct three experiments to evaluate the performance of our method in \eqref{prob:mainprob} under different data contamination scenarios. In Section~\ref{sec:reg_experiments} and Section~\ref{sec:nn_experiments}, we investigate linear regression and binary classification using synthetic data that embed outliers following a different data distribution and contain mislabeled classes in the training set, respectively.  Section~\ref{sec:rec_experiments} builds a recommender system on the Netflix Prize data \cite{bennett2007netflix} with random noise and under an adversarial attack.
In all experiments, we compare our method, denoted as SWAPE, with suitable existing benchmark methods, selected for each experiment. Further details on data generation, subproblem derivations, comparison methods, hyperparameter selection, and implementation for each experiment are provided in Appendices \ref{app:reg}-\ref{app:rec}. All tables presented in this section report average values over replications, with their standard deviations shown in parentheses.

\subsection{Linear regression with vertical and interior-clustered contaminations}
\label{sec:reg_experiments}
In this section, we analyze the performance of \eqref{prob:mainprob} on linear regression with potentially contaminated data. Given a data matrix \(A\in\mathbb{R}^{n\times d}\) and the response vector \(\h b\in\mathbb{R}^n\), the problem \eqref{prob:mainprob} can be written as
\begin{align}
\label{prob:mainreg}
\minimize_{\h x\in\mathbb{R}^d,\;\h w\in\mathbb{R}_+^n}\quad
F^{\rm reg}(\h x,\h w)
= \frac12\|\h w^{1/2}\odot (\h b-A\x)\|_2^2
+ \lambda R(\x)
+\gamma\sum_{i=1}^n \rho(w_i),
\end{align}
with the least squares loss in the first term. We choose $\rho$ to be the minimax concave penalty (MCP) \cite{zhang2010nearly} function, which is horizontally shifted to promote $w_i$ taking a value of 1. Since our setting considers a lower-dimensional problem, i.e., $d < n$, we set $R(\x) = 0$. We implement the
\eqref{alg:subproblems} to solve \eqref{prob:mainreg} using surrogates \eqref{MM:wc} and \eqref{MM:wc_W} for $\x$ and $\w$ subproblems, respectively. We derive a closed-form solution for both $\x$ and $\w$ subproblems so that each iteration can be computed efficiently.

SWAPE is compared to ordinary least squares (OLS), least absolute deviation (LAD), and Huber loss \cite{Huber1964} functions, as well as methods designed for training a model performing data selection, such as least trimmed squares (LTS) \cite{Rousseeuw1987} and $\Theta$-IPOD \cite{She2011}. The hyperparameters of our method and $\Theta$-IPOD are tuned using cross-validation, whereas LTS selects a fixed amount of samples at $75$\% to train the model.

\subsubsection{Experiment setting}
We test all models on three types of synthetic data: data without contamination, vertical contamination, and interior-clustered contamination. 
For the first case, we generate a matrix $A \in \mathbb{R}^{n \times d}$ whose entries follow \(A_{ij}\sim \mathcal{N}(\mu_A,\sigma_A^2)\). For this setting, we use
\(\mu_A=7.5\) and \(\sigma_A=4.0\).
Using the ground-truth vector \(\x^{ *} = 5\mathbf{1}_{d}\) and Gaussian noise $\bm{\varepsilon} \in \mathbb{R}^{n}$, the response vector \(\h b\in\mathbb{R}^{n}\) is generated according to
\begin{align}\label{data_generation:linear}
    \h b = A\x^{ *}+\bm{\varepsilon}
    \quad \text{ where } \quad
    \bm{\varepsilon} \sim \mathcal{N}(0,I_{n}).
\end{align}
For the vertical and interior-clustered contamination cases, we follow the data generation process in \cite{WISNOWSKI2001351} to include outlier points in the training set. In both settings, the training set contains $(1-\alpha)$ proportion of observations generated by \eqref{data_generation:linear} for some $\alpha \in (0,1)$ and an
$\alpha$ proportion following a different generation procedure, thereby synthetically creating data contamination. 
Our experiments vary the contamination level as $\alpha = 0.2, 0.3, 0.4$, 
from less contaminated to more contaminated cases. 
In addition, a parameter $\tau > 0$ is used to vary the distance between the (former) uncontaminated and (latter) contaminated data points, measuring the severity of contamination.
We include Figure~\ref{fig:residuals_true_fitted} to illustrate the two types of contamination considered in our experiments.
For each setting, we test on the severity levels $\tau = 5, 8$. We denote $\mathcal{I}_c$ as an index set that contains the described contaminated data points, i.e., $|\mathcal{I}_c| = \alpha n$. 

For each ground truth $\x^{ *}$, we generate a set of training instances $(A^{\rm tr}, \h b^{\rm tr}) \in \mathbb{R}^{600 \times 6} \times \mathbb{R}^{600}$ following the described procedure, and solve the resulting training problem \eqref{prob:mainreg} to obtain a solution $(\widehat{\x}, \widehat{\w})$. On the other hand, the testing set $(A^{\rm te}, \h b^{\rm te}) \in \mathbb{R}^{5000 \times 6} \times \mathbb{R}^{5000}$ is generated only using \eqref{data_generation:linear}. We conduct \(200\) independent replications for each setting, and report the results as means with standard deviations. 
\begin{figure}[t!]
    \centering    \includegraphics[width=0.9\linewidth]{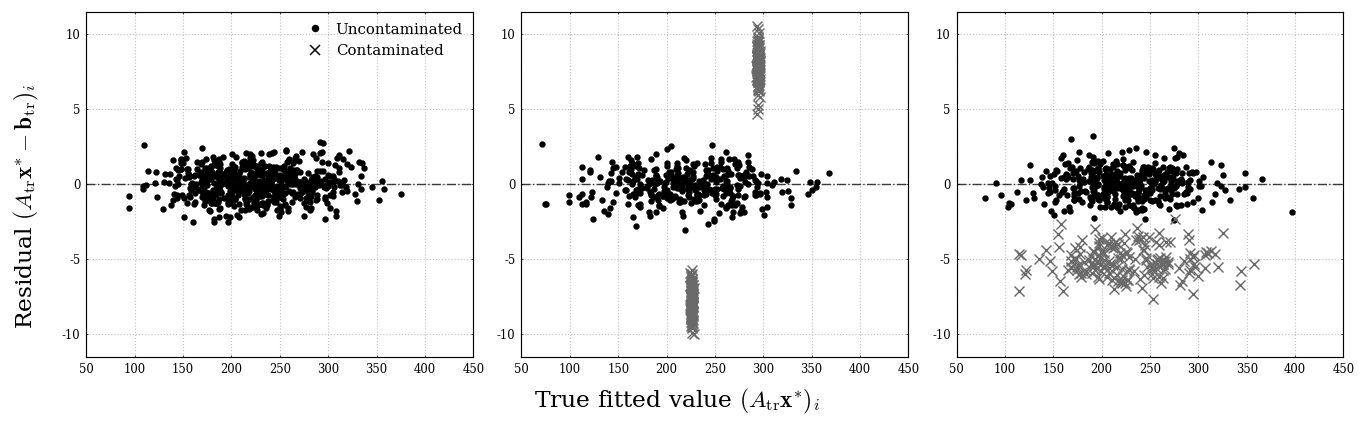}
    \caption{Residuals versus true fitted values from the first replication. From left to right: no contamination; interior-clustered contamination with two clouds for $\alpha=0.4$;  
    and vertical contamination with $\alpha=0.3$. 
    }
    \label{fig:residuals_true_fitted}
\end{figure}
To report the results, we consider the following three metrics to quantitatively evaluate the performance of the solution $\widehat{x}$ of each method:
\begin{align*}
\mathrm{MSE}
=
\frac{1}{5000}
\left\|
A^{\mathrm{te}}\widehat{\x}
-
{\h b}^{\mathrm{te}}
\right\|_2^2, \
\mathrm{Signal\ MSE} 
=
\frac{1}{5000}
\left\|
A^{\mathrm{te}}
\left(
\widehat{\x}-\x^{ *}
\right)
\right\|_2^2, 
\begin{array}{l@{}}
    \text{Estimation} \\
    \text{Error}
\end{array}
=
\left\|
\widehat{\x}-\x^{ *}
\right\|_2.
\end{align*}
The mean squared error (MSE) and signal MSE capture the prediction error of $\widehat{\x}$ measured by the response vector and ground truth, respectively. The estimation error computes the distance between the estimated solution and the ground truth.

In addition, we examine the effect of the weight, $\widehat{\w}$, in identifying and downweighting potentially contaminated data in the training set. Let $n_{\rm tr}$ denote the number of training data points. We compare the index set, $\mathcal{I}_c$, that includes the true contaminated data points, to the set of indices,
\begin{align} \label{set:Ihat} 
    \widehat{\mathcal{I}}_c \triangleq \{i \in \{ 1, \dots, n_{\rm tr} \} \ \vert \ \widehat{w_i} = 0 \}
\end{align}
which indicates the data points with zero weights.
Comparing the elements in the two sets allows us to compute 
\begin{align} \label{metric:recall}
\mathrm{Recall}
=
\frac{|\widehat{\mathcal{I}}_c\cap \mathcal{I}_c|}
{|\mathcal{I}_c|},
\quad
\mathrm{Precision}
=
\frac{|\widehat{\mathcal{I}}_c\cap \mathcal{I}_c|}
{|\widehat{\mathcal{I}}_c|}, 
\quad \text{and} \quad 
\mathrm{F1\text{-}score}
=
\frac{2\,\mathrm{Precision}\,\mathrm{Recall}}
{\mathrm{Precision}+\mathrm{Recall}}.
\end{align}
Precision is the fraction of observations identified as contaminated that are truly contaminated, while recall is the fraction of truly contaminated observations that are correctly identified.
The F1-score measures the overall performance by combining precision and recall.
Lastly, we report the percentage of data points utilized during our method's training, measured by
\begin{align} \label{metric:sparsity}
    \text{Retained } =  100 \left( 1 - \frac{| \widehat{\mathcal{I}_c} |}{n_{\rm tr}} \right) 
\end{align}
For LTS and \(\Theta\)-IPOD, the set $\widehat{\mathcal{I}_c}$ includes the data points omitted by the corresponding method. For LTS, the trimmed observations are identified as contaminated, whereas for \(\Theta\)-IPOD, an observation is identified as contaminated when its estimated mean shift parameter is nonzero.

\subsubsection{Results}
\label{sec:reg-results}

We report results for the case of no contamination in Table~\ref{tab:reg_clean}. All methods exhibit similar estimation performance, with OLS and Huber performing slightly better overall. We observe that SWAPE remains competitive despite retaining roughly half of the observations, while LTS trained with $75$\% of the data yields higher MSE, signal MSE, and estimation error. $\Theta$-IPOD achieves performance closer to OLS and Huber by using about $95$\% of the data.
\begingroup
    \renewcommand{\bfseries}{\relax} 
    
\begin{table}[H]
\scriptsize
\renewcommand{\arraystretch}{1.1}
\setlength{\tabcolsep}{4pt}
\centering
\caption{Results for linear regression with uncontaminated data. }
\label{tab:reg_clean}
\begin{tabular}{@{} l c c c c @{}}
\toprule
\textbf{Method} & \textbf{MSE} & \textbf{Signal MSE} & \textbf{Estimation Error} & \textbf{Retained} \\
\midrule
SWAPE & 1.020 (0.026) & 0.020 (0.014) & 0.030 (0.011) & 50.29 \\
$\Theta$-IPOD & 1.014 (0.024) & 0.014 (0.011) & 0.025 (0.009) & 94.98 \\
OLS & \textbf{1.010 (0.022)} & \textbf{0.010 (0.005)} & \textbf{0.022 (0.006)} & --  \\
LAD & 1.016 (0.024) & 0.016 (0.009) & 0.028 (0.008) & --  \\
Huber & 1.011 (0.023) & 0.011 (0.006) & 0.023 (0.007) & --  \\
LTS & 1.030 (0.029) & 0.030 (0.018) & 0.037 (0.012) & 75.00  \\
\bottomrule
\end{tabular}
\end{table}
\endgroup 

Table~\ref{tab:vertical_full} reports the results for the vertical contamination setting. Our method exhibits stable prediction accuracy across all contamination and severity settings, with MSE remaining close to $1.03$, while the benchmark methods' performance deteriorates sharply as the dataset embeds more vertical outliers. For instance, when $\tau = 5$, LTS and $\Theta$-IPOD obtain MSE values around $1.02$ for $\alpha = 0.2$. These values increase to $2.01$ and $4.21$, respectively, when the contamination level rises to $\alpha = 0.4$. The different performance of our model and the two benchmark models can be attributed to the percentage of retained data. The three models use about $74.1$--$79.8$\% of data and achieve high F1 scores between $0.86$--$0.95$ when $\alpha = 0.2$. As $\alpha$ values increase to $0.3$ and $0.4$, our method adaptively uses $64.7$\% and $54.4$\% of the data points. On the other hand, while $\Theta$-IPOD adapts and utilizes $70.4$\% of data for $\alpha = 0.3$, but its retained value of $94.7$\% significantly exceeds the nominal contamination rate when $\alpha = 0.4$. Using a fixed $75$\% of the data, LTS performs comparably only when the contamination level matches $\alpha = 0.3$. Consequently, the recall values for $\Theta$-IPOD and LTS drop substantially for $\alpha = 0.4$, resulting in F1 scores of $0.21$ and $0.76$, respectively. In contrast, our method achieves an F1 score of $0.93$. OLS is highly sensitive to vertical contamination. LAD and Huber offer greater robustness than OLS but consistently produce larger prediction and estimation errors than the SWAPE. We observe a similar trend for $\tau = 8$ when the severity of vertical contamination increases.

\begingroup 
    \renewcommand{\bfseries}{\relax} 
    
\begin{table}[h!]
\scriptsize
\centering
\renewcommand{\arraystretch}{1.1}
\setlength{\tabcolsep}{3pt}
\caption{Results for linear regression on data with vertical contamination. Parameters $\alpha$ and $\tau$, indicating the percentage and severity of contamination, are annotated for each scenario.}
\label{tab:vertical_full}
\begin{tabular}{@{}l l c c c c c c c@{}}
\toprule
\textbf{Setting} &
\textbf{Method} &
\textbf{MSE} &
\textbf{\begin{tabular}[c]{@{}c@{}}Signal\\MSE\end{tabular}} &
\textbf{\begin{tabular}[c]{@{}c@{}}Estimation\\Error\end{tabular}} &
\textbf{Retained} &
\textbf{Precision} &
\textbf{Recall} &
\textbf{F1} \\
\midrule
\multirow{6}{*}{\shortstack[l]{$\alpha = 0.2$\\$\tau = 5$}} & SWAPE & 1.033 (0.038) & 0.033 (0.032) & 0.034 (0.013) &74.12 & 0.804 (0.189) & 0.964 (0.064) & 0.859 (0.111) \\
 & $\Theta$-IPOD & 1.024 (0.031) & 0.023 (0.023) & \textbf{0.030 (0.011)} & 79.79 & \textbf{0.953 (0.060)} & 0.957 (0.062) & \textbf{0.952 (0.036)} \\
 & OLS & 2.003 (0.105) & 1.001 (0.099) & 0.072 (0.010) & -- & -- & -- & -- \\
 & LAD & 1.131 (0.050) & 0.130 (0.045) & 0.040 (0.011) & -- & -- & -- & -- \\
 & Huber & 1.193 (0.049) & 0.192 (0.046) & 0.041 (0.009) & -- & -- & -- & -- \\
 & LTS & \textbf{1.020 (0.023)} & \textbf{0.020 (0.012)} & 0.030 (0.010) & 75.00 & 0.799 (0.003) & \textbf{0.999 (0.003)} & 0.888 (0.003) \\
\cmidrule(lr){1-9}
\multirow{6}{*}{\shortstack[l]{$\alpha = 0.3$\\$\tau = 5$}} & SWAPE & 1.035 (0.035) & 0.034 (0.027) & 0.037 (0.014) & 64.74 & 0.862 (0.139) & \textbf{0.980 (0.038)} & 0.909 (0.078) \\
 & $\Theta$-IPOD & \textbf{1.021 (0.027)} & \textbf{0.021 (0.016)} & \textbf{0.030 (0.011)} & 70.38 & 0.987 (0.011) & 0.974 (0.029) & \textbf{0.980 (0.014)} \\
 & OLS & 3.202 (0.155) & 2.201 (0.150) & 0.096 (0.010) & -- & -- & -- & -- \\
 & LAD & 1.353 (0.085) & 0.352 (0.082) & 0.054 (0.012) & -- & -- & -- & -- \\
 & Huber & 1.960 (0.153) & 0.959 (0.148) & 0.075 (0.014) & -- & -- & -- & -- \\
 & LTS & 1.104 (0.043) & 0.103 (0.036) & 0.054 (0.016) & 75.00  & \textbf{0.999 (0.003)} & 0.833 (0.003) & 0.908 (0.003) \\
\cmidrule(lr){1-9}
\multirow{6}{*}{\shortstack[l]{$\alpha = 0.4$\\$\tau = 5$}} & SWAPE & \textbf{1.033 (0.032)} & \textbf{0.032 (0.023)} & \textbf{0.037 (0.014)} & 54.40 & 0.885 (0.105) & \textbf{0.992 (0.016)} & \textbf{0.931 (0.059)} \\
 & $\Theta$-IPOD & 4.208 (0.468) & 3.205 (0.468) & 0.131 (0.018) & 94.70 &  0.834 (0.306) & 0.125 (0.077) & 0.213 (0.122) \\
 & OLS & 4.865 (0.213) & 3.862 (0.201) & 0.119 (0.010) & -- & -- & -- & -- \\
 & LAD & 2.108 (0.235) & 1.106 (0.233) & 0.091 (0.021) & -- & -- & -- & -- \\
 & Huber & 4.438 (0.222) & 3.435 (0.212) & 0.120 (0.012) & -- & -- & -- & -- \\
 & LTS & 2.010 (0.221) & 1.008 (0.218) & 0.153 (0.035) & 75.00  & \textbf{0.987 (0.014)} & 0.617 (0.009) & 0.759 (0.011) \\ 
\midrule \midrule
\multirow{6}{*}{\shortstack[l]{$\alpha = 0.20$\\$\tau = 8$}} & SWAPE & 1.033 (0.033) & 0.032 (0.027) & 0.035 (0.015) & 71.79 & 0.770 (0.210) & 0.991 (0.024) & 0.848 (0.139) \\
 & $\Theta$-IPOD & \textbf{1.015 (0.023)} & \textbf{0.015 (0.009)} & \textbf{0.027 (0.009)} & 78.71 & \textbf{0.947 (0.062)} & 1.000 (0.002) & \textbf{0.972 (0.034)} \\
 & OLS & 3.565 (0.184) & 2.564 (0.178) & 0.114 (0.017) & -- & -- & -- & -- \\
 & LAD & 1.134 (0.045) & 0.133 (0.041) & 0.041 (0.011) & -- & -- & -- & -- \\
 & Huber & 1.198 (0.050) & 0.197 (0.045) & 0.043 (0.010) & -- & -- & -- & -- \\
 & LTS & 1.019 (0.024) & 0.019 (0.010) & 0.030 (0.009) & 75.00  & 0.800 (0.000) & \textbf{1.000 (0.000)} & 0.889 (0.000) \\
\cmidrule(lr){1-9}
\multirow{6}{*}{\shortstack[l]{$\alpha = 0.30$\\$\tau = 8$}} & SWAPE & 1.031 (0.030) & 0.030 (0.023) & 0.035 (0.013) & 63.69 & 0.847 (0.134) & 0.996 (0.012) & 0.909 (0.081) \\
 & $\Theta$-IPOD & \textbf{1.016 (0.022)} & \textbf{0.015 (0.009)} & \textbf{0.026 (0.009)} & 69.81 & 0.994 (0.009) & \textbf{1.000 (0.002)} & \textbf{0.997 (0.005)} \\
 & OLS & 6.646 (0.282) & 5.643 (0.276) & 0.153 (0.018) & -- & -- & -- & -- \\
 & LAD & 1.362 (0.084) & 0.361 (0.081) & 0.056 (0.013) & -- & -- & -- & -- \\
 & Huber & 2.032 (0.173) & 1.030 (0.170) & 0.079 (0.017) & -- & -- & -- & -- \\
 & LTS & 1.380 (0.111) & 0.377 (0.109) & 0.108 (0.027) & 75.00  & \textbf{1.000 (0.000)} & 0.833 (0.000) & 0.909 (0.000) \\
\cmidrule(lr){1-9}
\multirow{6}{*}{\shortstack[l]{$\alpha = 0.40$\\$\tau = 8$}} & SWAPE & \textbf{1.032 (0.031)} & \textbf{0.032 (0.022)} & \textbf{0.038 (0.013)} & 54.58 & 0.891 (0.098) & \textbf{0.999 (0.005)} & \textbf{0.939 (0.056)} \\
 & $\Theta$-IPOD & 4.725 (3.644) & 3.726 (3.642) & 0.171 (0.137) & 75.19 & 0.983 (0.074) & 0.618 (0.371) & 0.693 (0.306) \\
 & OLS & 10.950 (0.413) & 9.947 (0.410) & 0.191 (0.014) & -- & -- & -- & -- \\
 & LAD & 2.144 (0.272) & 1.143 (0.270) & 0.092 (0.024) & -- & -- & -- & -- \\
 & Huber & 10.156 (0.422) & 9.152 (0.421) & 0.193 (0.019) & -- & -- & -- & -- \\
 & LTS & 4.310 (0.438) & 3.309 (0.435) & 0.318 (0.040) & 75.00  & \textbf{0.997 (0.007)} & 0.623 (0.004) & 0.767 (0.005) \\
\bottomrule
\end{tabular}
\end{table}
\endgroup

Lastly, Table~\ref{tab:interior_full} summarizes the performance under interior-clustered contamination. We first discuss the lower-severity case at $\tau = 5$. When $\alpha = 0.2$, SWAPE, $\Theta$-IPOD, and LTS outperform the other regression methods, with MSE values around $1.02$ while retaining $71.7$-$78.5$\% of the data. As $\alpha$ increases to $0.3$ and $0.4$, our model reduces the retained data to $63.0$\% and $56.6$\% while maintaining strong prediction performance across metrics. The LAD also shows robust performance across all levels of contamination that is comparable to the SWAPE.
In contrast, the performance of $\Theta$-IPOD and LTS declines as $\alpha$ increases and eventually yields higher MSE, signal MSE, and estimation error than all regression benchmark methods at $\alpha = 0.4$. We observe that our model's performance improves when $\tau = 8$, while the other methods seem to suffer in all prediction metrics.

\begingroup 
    \renewcommand{\bfseries}{\relax} 
    
\begin{table}[!htbp]
\tiny
\centering
\renewcommand{\arraystretch}{1.0}
\setlength{\tabcolsep}{3pt}
\caption{\small Results for linear regression on data with interior-clustered contamination. Parameters $\alpha$ and $\tau$, indicating the percentage and severity of contamination, are annotated for each scenario.}
\label{tab:interior_full}

\begin{tabular}{@{}l l c c c c c c c@{}}
\toprule
\textbf{Setting} & \textbf{Method} & \textbf{MSE} & \textbf{\begin{tabular}[c]{@{}c@{}}Signal\\MSE\end{tabular}} &
\textbf{\begin{tabular}[c]{@{}c@{}}Estimation\\Error\end{tabular}} &
\textbf{Retained} & \textbf{Precision} & \textbf{Recall} & \textbf{F1} \\
\midrule
\multirow{6}{*}{\shortstack[l]{$\alpha = 0.2$\\$\tau = 5$\\1 cloud}} & SWAPE & \textbf{1.084 (0.307)} & \textbf{0.084 (0.311)} & 0.044 (0.050) & 70.86 & 0.732 (0.222) & 0.967 (0.121) & 0.814 (0.167) \\
 & $\Theta$-IPOD & 1.127 (0.414) & 0.126 (0.413) & 0.045 (0.068) & 79.08 & \textbf{0.862 (0.185)} & 0.928 (0.233) & \textbf{0.887 (0.216)} \\
 & OLS & 2.503 (0.180) & 1.500 (0.173) & 0.247 (0.040) & -- & -- & -- & -- \\
 & LAD & 1.529 (0.360) & 0.527 (0.356) & 0.145 (0.060) & -- & -- & -- & -- \\
 & Huber & 1.854 (0.421) & 0.852 (0.416) & 0.186 (0.064) & -- & -- & -- & -- \\
 & LTS & 1.100 (0.414) & 0.100 (0.415) & \textbf{0.041 (0.061)} & 75.00 & 0.774 (0.127) & \textbf{0.967 (0.159)} & 0.860 (0.141) \\
\cmidrule(lr){1-9}
\multirow{6}{*}{\shortstack[l]{$\alpha = 0.3$\\$\tau = 5$\\1 cloud}} & SWAPE & \textbf{3.014 (2.611)} & \textbf{2.013 (2.609)} & \textbf{0.220 (0.230)} & 61.94 & \textbf{0.520 (0.349)} & \textbf{0.667 (0.421)} & \textbf{0.578 (0.374)} \\
 & $\Theta$-IPOD & 4.262 (1.634) & 3.258 (1.633) & 0.358 (0.152) & 87.79 & 0.238 (0.351) & 0.162 (0.333) & 0.175 (0.332) \\
 & OLS & 3.611 (0.276) & 2.605 (0.270) & 0.323 (0.042) & -- & -- & -- & -- \\
 & LAD & 3.463 (1.137) & 2.458 (1.139) & 0.318 (0.105) & -- & -- & -- & -- \\
 & Huber & 3.719 (0.504) & 2.713 (0.504) & 0.336 (0.061) & -- & -- & -- & -- \\
 & LTS & 4.282 (2.889) & 3.279 (2.890) & 0.328 (0.233) & 75.00 & 0.452 (0.470) & 0.377 (0.392) & 0.411 (0.427) \\
\cmidrule(lr){1-9}
\multirow{6}{*}{\shortstack[l]{$\alpha = 0.4$\\$\tau = 5$\\1 cloud}} & SWAPE & 8.825 (3.813) & 7.825 (3.812) & 0.576 (0.173) & 60.54 & \textbf{0.175 (0.189)} & \textbf{0.204 (0.238)} & \textbf{0.186 (0.210)} \\
 & $\Theta$-IPOD & 7.979 (2.301) & 6.977 (2.299) & 0.558 (0.098) & 82.98 & 0.034 (0.060) & 0.011 (0.009) & 0.016 (0.013) \\
 & OLS & \textbf{4.692 (0.494)} & \textbf{3.690 (0.489)} & \textbf{0.383 (0.041)} & -- & -- & -- & -- \\
 & LAD & 6.462 (1.251) & 5.462 (1.252) & 0.483 (0.058) & -- & -- & -- & -- \\
 & Huber & 5.522 (0.662) & 4.521 (0.660) & 0.435 (0.046) & -- & -- & -- & -- \\
 & LTS & 9.264 (3.291) & 8.262 (3.295) & 0.609 (0.112) & 75.00 & 0.052 (0.100) & 0.032 (0.062) & 0.040 (0.077) \\
\midrule \midrule
\multirow{6}{*}{\shortstack[l]{$\alpha = 0.20$\\$\tau = 8$\\1 cloud}} & SWAPE & \textbf{1.029 (0.031)} & \textbf{0.029 (0.023)} & \textbf{0.035 (0.014)} & 70.94 & 0.740 (0.192) & \textbf{0.997 (0.015)} & 0.835 (0.130) \\
 & $\Theta$-IPOD & 1.367 (1.252) & 0.367 (1.253) & 0.062 (0.126) & 79.15 & \textbf{0.846 (0.238)} & 0.927 (0.258) & \textbf{0.883 (0.246)} \\
 & OLS & 4.856 (0.436) & 3.855 (0.431) & 0.402 (0.059) & -- & -- & -- & -- \\
 & LAD & 1.638 (0.659) & 0.638 (0.663) & 0.158 (0.076) & -- & -- & -- & -- \\
 & Huber & 2.326 (1.108) & 1.326 (1.110) & 0.227 (0.110) & -- & -- & -- & -- \\
 & LTS & 1.110 (0.740) & 0.109 (0.744) & 0.037 (0.065) & 75.00 & 0.788 (0.095) & 0.985 (0.119) & 0.876 (0.106) \\
\cmidrule(lr){1-9}
\multirow{6}{*}{\shortstack[l]{$\alpha = 0.30$\\$\tau = 8$\\1 cloud}} & SWAPE & \textbf{2.136 (3.274)} & \textbf{1.135 (3.272)} & \textbf{0.114 (0.217)} & 58.10 & \textbf{0.691 (0.238)} & \textbf{0.920 (0.249)} & \textbf{0.782 (0.236)} \\
 & $\Theta$-IPOD & 10.359 (4.574) & 9.358 (4.567) & 0.614 (0.260) & 86.21 & 0.136 (0.328) & 0.127 (0.331) & 0.128 (0.329) \\
 & OLS & 7.605 (0.626) & 6.603 (0.626) & 0.514 (0.067) & -- & -- & -- & -- \\
 & LAD & 6.778 (4.137) & 5.777 (4.139) & 0.473 (0.232) & -- & -- & -- & -- \\
 & Huber & 8.095 (1.584) & 7.093 (1.581) & 0.545 (0.112) & -- & -- & -- & -- \\
 & LTS & 9.242 (7.275) & 8.243 (7.275) & 0.524 (0.372) & 75.00& 0.437 (0.495) & 0.364 (0.413) & 0.397 (0.450)\\
\cmidrule(lr){1-9}
\multirow{6}{*}{\shortstack[l]{$\alpha = 0.40$\\$\tau = 8$\\1 cloud}} & SWAPE & 12.408 (9.314) & 11.416 (9.321) & \textbf{0.606 (0.432)} & 55.17 & \textbf{0.364 (0.371)} & \textbf{0.440 (0.433)} & \textbf{0.396 (0.395)} \\
 & $\Theta$-IPOD & 21.286 (7.938) & 20.299 (7.933) & 0.956 (0.200) & 76.93 & 0.003 (0.005) & 0.002 (0.003) & 0.002 (0.004) \\
 & OLS & \textbf{10.372 (1.316)} & \textbf{9.380 (1.312)} & 0.615 (0.065) & -- & -- & -- & -- \\
 & LAD & 17.180 (5.163) & 16.191 (5.161) & 0.835 (0.137) & -- & -- & -- & -- \\
 & Huber & 13.066 (1.908) & 12.076 (1.904) & 0.719 (0.076) & -- & -- & -- & -- \\
 & LTS & 22.914 (8.380) & 21.928 (8.375) & 1.001 (0.164) & 75.00 & 0.003 (0.007) & 0.002 (0.004) & 0.003 (0.005) \\
\midrule \midrule 
\multirow{6}{*}{\shortstack[l]{$\alpha = 0.20$\\$\tau = 5$\\2 clouds}} & SWAPE & 1.031 (0.034) & 0.031 (0.027) & 0.037 (0.017) & 71.75 & 0.757 (0.202) & 0.976 (0.057) & 0.833 (0.128) \\
 & $\Theta$-IPOD & \textbf{1.018 (0.023)} & \textbf{0.017 (0.010)} & \textbf{0.029 (0.010)} & 78.49 & \textbf{0.916 (0.081)} & 0.976 (0.042) & \textbf{0.942 (0.038)} \\
 & OLS & 1.947 (0.383) & 0.948 (0.382) & 0.235 (0.053) & -- & -- & -- & -- \\
 & LAD & 1.162 (0.127) & 0.161 (0.126) & 0.092 (0.036) & -- & -- & -- & -- \\
 & Huber & 1.254 (0.215) & 0.254 (0.215) & 0.115 (0.048) & -- & -- & -- & -- \\
 & LTS & 1.020 (0.023) & 0.019 (0.010) & 0.030 (0.009) & 75.00 & 0.799 (0.003) & \textbf{0.999 (0.003)} & 0.888 (0.003) \\
\cmidrule(lr){1-9}
\multirow{6}{*}{\shortstack[l]{$\alpha = 0.30$\\$\tau = 5$\\2 clouds}} & SWAPE & \textbf{1.092 (0.491)} & \textbf{0.092 (0.491)} & \textbf{0.045 (0.058)} & 63.01 & 0.820 (0.155) & \textbf{0.979 (0.090)} & 0.885 (0.113) \\
 & $\Theta$-IPOD & 1.557 (1.258) & 0.557 (1.260) & 0.100 (0.154) & 73.40 &  0.842 (0.308) & 0.835 (0.336) & 0.833 (0.323)\\
 & OLS & 3.068 (0.593) & 2.065 (0.589) & 0.351 (0.056) & -- & -- & -- & -- \\
 & LAD & 1.674 (0.678) & 0.673 (0.679) & 0.184 (0.088) & -- & -- & -- & -- \\
 & Huber & 2.259 (0.864) & 1.258 (0.864) & 0.260 (0.100) & -- & -- & -- & -- \\
 & LTS & 1.354 (1.159) & 0.353 (1.159) & 0.099 (0.103) & 75.00 & \textbf{0.961 (0.167)} & 0.801 (0.139) & 0.874 (0.152) \\
\cmidrule(lr){1-9}
\multirow{6}{*}{\shortstack[l]{$\alpha = 0.40$\\$\tau = 5$\\2 clouds}} & SWAPE & 4.134 (5.712) & 3.132 (5.703) & \textbf{0.248 (0.361)} & 56.58 & \textbf{0.664 (0.357)} & \textbf{0.758 (0.400)} & \textbf{0.704 (0.372)} \\
 & $\Theta$-IPOD & 6.657 (4.203) & 5.655 (4.197) & 0.533 (0.237) & 86.58 & 0.349 (0.433) & 0.171 (0.286) & 0.210 (0.322) \\
 & OLS & 4.585 (0.756) & 3.584 (0.757) & 0.465 (0.056) & -- & -- & -- & -- \\
 & LAD & \textbf{3.975 (2.532)} & \textbf{2.974 (2.533)} & 0.384 (0.190) & -- & -- & -- & -- \\
 & Huber & 4.661 (1.436) & 3.660 (1.437) & 0.462 (0.107) & -- & -- & -- & -- \\
 & LTS & 7.527 (6.219) & 6.524 (6.210) & 0.530 (0.311) & 75.00 & 0.559 (0.436) & 0.350 (0.272) & 0.430 (0.335) \\
\midrule \midrule 
\multirow{6}{*}{\shortstack[l]{$\alpha = 0.20$\\$\tau = 8$\\2 clouds}} & SWAPE & 1.030 (0.032) & 0.029 (0.026) & 0.036 (0.016) & 72.39 & 0.778 (0.201) & 0.990 (0.027) & 0.855 (0.131) \\
 & $\Theta$-IPOD & \textbf{1.016 (0.022)} & \textbf{0.015 (0.009)} & \textbf{0.027 (0.009)} & 78.66 & \textbf{0.943 (0.073)} & \textbf{1.000 (0.000)} & \textbf{0.969 (0.040)} \\
 & OLS & 3.357 (0.954) & 2.359 (0.952) & 0.371 (0.084) & -- & -- & -- & -- \\
 & LAD & 1.153 (0.120) & 0.153 (0.117) & 0.090 (0.035) & -- & -- & -- & -- \\
 & Huber & 1.246 (0.246) & 0.246 (0.244) & 0.112 (0.050) & -- & -- & -- & -- \\
 & LTS & 1.020 (0.024) & 0.019 (0.011) & 0.030 (0.010) & 75.00 & 0.800 (0.000) & \textbf{1.000 (0.000)} & 0.889 (0.000) \\
\cmidrule(lr){1-9}
\multirow{6}{*}{\shortstack[l]{$\alpha = 0.30$\\$\tau = 8$\\2 clouds}} & SWAPE & \textbf{1.029 (0.033)} & \textbf{0.028 (0.024)} & \textbf{0.035 (0.016)} & 63.59 & 0.846 (0.135) & \textbf{0.998 (0.009)} & \textbf{0.909 (0.082)} \\
 & $\Theta$-IPOD & 3.297 (4.317) & 2.301 (4.324) & 0.206 (0.315) & 75.62 & 0.750 (0.400) & 0.754 (0.425) & 0.747 (0.416) \\
 & OLS & 6.315 (1.464) & 5.315 (1.466) & 0.564 (0.086) & -- & -- & -- & -- \\
 & LAD & 1.924 (1.518) & 0.923 (1.516) & 0.200 (0.131) & -- & -- & -- & -- \\
 & Huber & 3.487 (2.661) & 2.489 (2.667) & 0.337 (0.201) & -- & -- & -- & -- \\
 & LTS & 2.087 (2.828) & 1.086 (2.826) & 0.186 (0.160) & 75.00 & \textbf{0.962 (0.169)} & 0.802 (0.140) & 0.875 (0.153) \\
\cmidrule(lr){1-9}
\multirow{6}{*}{\shortstack[l]{$\alpha = 0.40$\\$\tau = 8$\\2 clouds}} & SWAPE & \textbf{2.521 (5.909)} & \textbf{1.518 (5.897)} & \textbf{0.113 (0.281)} & 54.31 & \textbf{0.826 (0.224)} & \textbf{0.940 (0.226)} & \textbf{0.876 (0.219)} \\
 & $\Theta$-IPOD & 18.433 (14.803) & 17.409 (14.776) & 0.910 (0.457) & 84.59 & 0.270 (0.432) & 0.152 (0.308) & 0.174 (0.333)\\
 & OLS & 10.193 (2.077) & 9.177 (2.062) & 0.741 (0.092) & -- & -- & -- & -- \\
 & LAD & 8.090 (8.375) & 7.075 (8.364) & 0.536 (0.385) & -- & -- & -- & -- \\
 & Huber & 10.407 (4.471) & 9.390 (4.455) & 0.725 (0.221) & -- & -- & -- & -- \\
 & LTS & 18.487 (16.341) & 17.465 (16.303) & 0.872 (0.492) &75.00 & 0.549 (0.427) & 0.343 (0.267) & 0.422 (0.328) \\
\bottomrule
\end{tabular}
\end{table}
\endgroup

\subsection{Binary classification with flipped labels}
\label{sec:nn_experiments}

This experiment examines binary classification on the synthetic two-moons dataset, available in the Scikit-learn library \cite{scikit-learn}. We use a fully connected, two-hidden-layer feed-forward neural network with softplus activation and sigmoid functions for the hidden and output neurons, respectively, to compute the probability that a data point belongs to one of two classes. We embed the parameters of the network, e.g., edge weights and bias terms, in the variable $\x$, and denote $p_i(\x)$ as the probability of an $i$-th data point belonging to a class $b_i \in \{0,1\}$. Applying the cross-entropy loss, the resulting problem is written as
\begin{align}
\label{prob:mainnn}
\minimize_{\x\in\mathbb{R}^d,\;\w\in\mathbb{R}_+^n}\quad
-\sum_{i=1}^{n} w_i
\left[
b_i\log  p_i(\x)
+
(1- b_i)
\log\!\left(1- p_i(\x)\right)
\right]+\lambda R(\x)
+\gamma\sum_{i=1}^{n}\rho(w_i).
\end{align}
We choose a variant of LASSO \cite{tibshirani96lasso} for $\rho$ and two types of regularizers for $R(\x)$, which will be discussed shortly. 
Our experiment applies 
\eqref{alg:subproblems} to solve \eqref{prob:mainnn} using surrogates \eqref{MM:wc} and \eqref{MM:wc_W} for $\x$ and $\w$ subproblems, respectively. 
We note that the choices of differentiable activation functions used in our experiments result in a weakly convex $\x$-subproblem, provided that $\x$ belongs to a compact constraint set, which in turn helps convexify the problem. However, we do not impose any restriction on $\x$ in our implementation and rely on an existing solver, Adam \cite{kingma2015adam}, to find a stationary solution for the $\x$-subproblem. On the other hand, the $\w$-subproblem is updated using a closed-form solution.

We compare our method to a baseline neural network that solves \eqref{prob:mainnn} only for $\x$ while fixing \(\w=\mathbf{1}_n\). For both the baseline and our models, we consider four training variants: no regularization where $R(\x) = 0$, \(\ell_2\) regularization \cite{krogh1991simple} which penalizes large network parameters by using $R(\x) = \tfrac12 \| \x \|_2^2$, 
early stopping \cite{bengio2012practical, prechelt2002early} which terminates training when validation performance no longer improves, and dropout \cite{srivastava2014dropout} which randomly deactivates hidden units during training. The SWAPE and baseline models use the same network architecture and training configuration within each variant.

\subsubsection{Experiment setting}
\label{sec: nnsetting}

Following the two-moons model \cite{scikit-learn}, we generate two featured data points, each belonging to class $0$ or $1$. We generate a training set $A^{\rm tr} \in \mathbb{R}^{500 \times 2}$ and label $\h b^{\rm tr} \in \{0,1\}^{500}$, and a testing set $A^{\rm te} \in \mathbb{R}^{5000 \times 2}$ and label $\h b^{\rm te} \in \{0,1\}^{5000}$.
This experiment considers a setting in which a subset of training data points has flipped labels. If the $i$-th data point belongs to a set of contaminated indices, e.g., $i \in \mathcal{I}_c$, then we flip and reassign the label of the data point, e.g., from class $0$ to $1$ or vice versa, to introduce label contamination. The test labels remain unchanged. For each contamination proportion $\alpha\in\{0.25,0.30,0.35,0.40,0.45\}$, we conduct $200$ independent replications. Figure~\ref{fig:nn-data2} illustrates the training and test data from a replication with \(\alpha=0.30\). 

\begin{figure}[t]
    \centering    \includegraphics[width=0.92\textwidth]{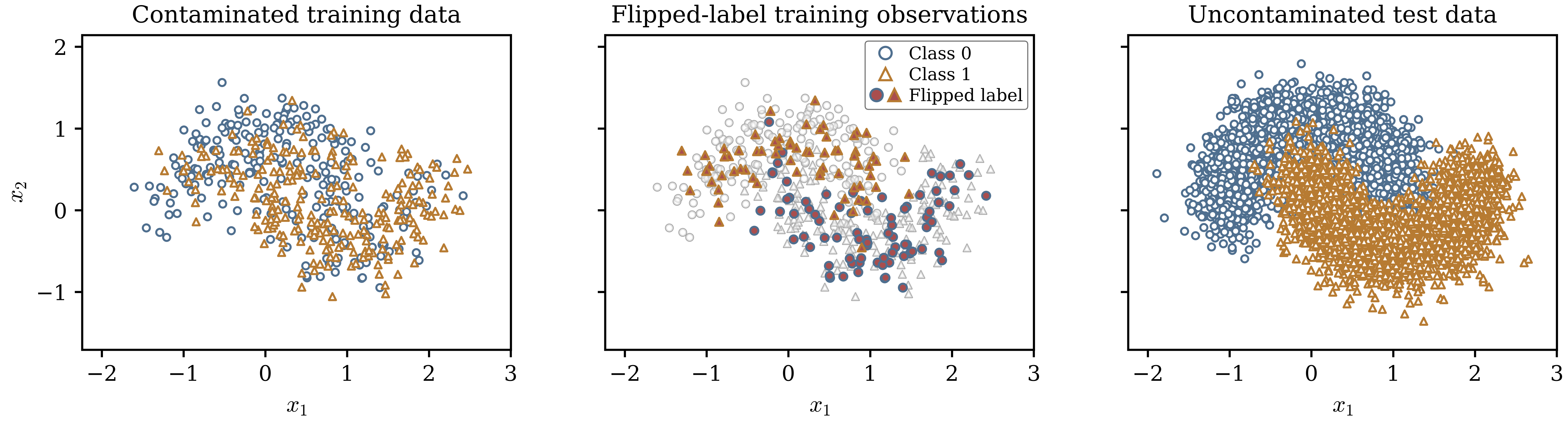}

    \caption{Representative training and test datasets for
    \(\alpha=0.30\) in a replication. 
    Left: the contaminated training set, where circles and triangles denote the observed class labels. Center: the label contamination in the training set, with flipped-label observations highlighted in red and the remaining observations shown in gray. Right: the uncontaminated test set with the true class labels.}
    \label{fig:nn-data2}
\end{figure}
Let us denote $(\widehat{\x}, \widehat{\w})$ as the solution of \eqref{prob:mainnn} given by the training set. Using $\widehat{\x}$, we predict the $i$-th testing data point to belong to class $1$ if $\widehat{p}_i = p_i(\widehat{\x}) \geq 0.5$, and class $0$ otherwise. The resulting predicted label is denoted as $\widehat{b}_i^{\rm te}$. To compare the model performance, we measure the prediction accuracy and the area under the receiver operating characteristic curve (AUC) defined as: 
\begin{align*}
    \text{Accuracy } = \frac{1}{5000} \sum \limits_{i=1}^{5000} \mathbb{I}(b_i^{\rm te} = \widehat{b}^{\rm te}_i) 
    \ \text{and} \ 
    \text{AUC } = \frac{1}{n_0 \, n_1}
    \sum \limits_{i: b_i^{\rm te} = 1}
    \sum \limits_{j: b_j^{\rm te} = 0}
    \left[
    \mathbb{I}(\widehat{p}_i > \widehat{p}_j)
    +\frac{1}{2}
    \mathbb{I}(\widehat{p}_i = \widehat{p}_j)
    \right],
\end{align*}
where $n_0$ and $n_1$ are the number of observations in $\h b^{\rm te}$ in class $0$ and $1$, respectively. The indicator $\mathbb{I}(\bullet)$ returns $1$ if the condition in the parentheses holds. Our results also report the effects of $\widehat{\w}$ in informing data points with flipped labels. Following the linear regression experiments, we report precision, recall, and F1 score, defined in \eqref{metric:recall}, and the retained percentage of training data in \eqref{metric:sparsity}.

\subsubsection{Results}
Table~\ref{tab:nn_classification_performance} reports the classification performance on the test set.
SWAPE outperforms the baseline neural network under most variants and contamination levels. Both accuracy and AUC improve significantly under no additional regularization, \(\ell_2\) regularization, and dropout, while the two models perform similarly under early stopping. The gains are largest without regularization and under \(\ell_2\) regularization, whereas in dropout, the improvement is smaller but remains consistent across contamination levels. For the above three variants, our method's performance is achieved by retaining $73.34$--$76.19$\% of training data when $\alpha = 0.25$. When the contamination level increases to $\alpha = 0.45$, the percentage of retained data decreases to $61.78$--$65.28$\%, reducing the influence of mislabeled training observations.
The precision, recall, and F1 score show that the downweighted observations align reasonably well with the mislabeled samples, with the highest F1 score of $0.828$ obtained by the no-regularization method when $\alpha = 0.25$. As $\alpha$ increases, the F1 scores decrease, indicating the difficulty of identifying flipped labels in the case of a high level of contamination.

In contrast, although early stopping improves the baseline method's accuracy and AUC, this variant limits learning of the network parameters and observation weights in our method. As a result, the trained weights remain close to their initialization, \(\w^0=\mathbf{1}_n\), making early stopping less suitable for SWAPE. For all values of $\alpha$, early stopping uses the entire dataset, i.e., retaining $100$\% of the training data.

Figure~\ref{fig:nn-boundaries} illustrates an example decision boundary of SWAPE and baseline models. In no regularization, \(\ell_2\) regularization, and dropout, our method recovers decision boundaries that align with the two-moons structure, whereas the corresponding baseline boundaries are more irregular. This suggests that the learned weights reduce the influence of mislabeled observations on the decision boundary, allowing the network to better recover the underlying class structure.
\begin{figure}[t!]
    \centering
    \includegraphics[width=\textwidth]{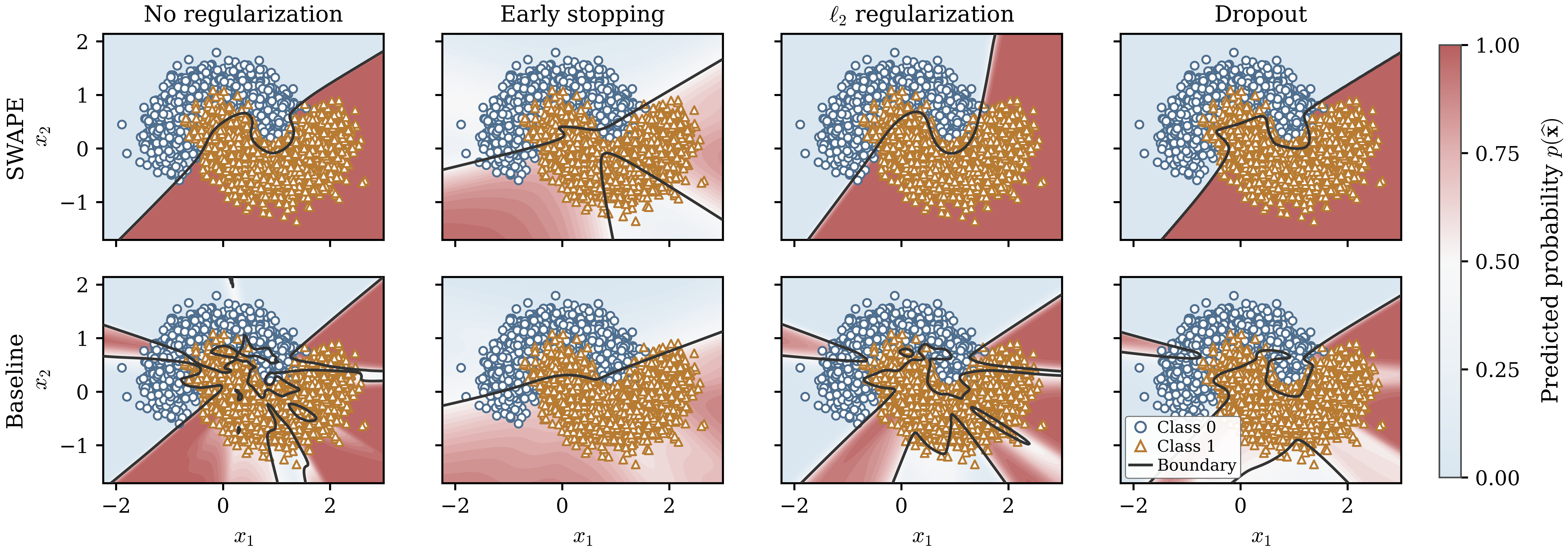}
    \caption{Estimated decision boundaries for \(\alpha=0.30\) in the same replication as Figure~\ref{fig:nn-data2}.}
    \label{fig:nn-boundaries}
\end{figure}

\begingroup
    \renewcommand{\bfseries}{\relax} 
\begin{table}[htbp]
\centering
\caption{Results for binary classification on data with flipped label contamination 
}
\label{tab:nn_classification_performance}
\scriptsize
\setlength{\tabcolsep}{3.5pt}
\renewcommand{\arraystretch}{1.1}
\resizebox{\textwidth}{!}{
\begin{tabular}{cc cc cc cccc}
\toprule
\multirow{2}{*}{$\alpha$} & \multirow{2}{*}{\textbf{Regularization}} & \multicolumn{2}{c}{\textbf{Accuracy}} & \multicolumn{2}{c}{\textbf{AUC}} & \multicolumn{4}{c}{\textbf{Weight Trained by SWAPE}} \\
\cmidrule(lr){3-4}\cmidrule(lr){5-6}\cmidrule(lr){7-10}
& & \textbf{Baseline} & \textbf{SWAPE} & \textbf{Baseline} & \textbf{SWAPE} & \textbf{Retained} & \textbf{Precision} & \textbf{Recall} & \textbf{F1} \\
\midrule
\multirow{4}{*}{0.25} & None & 0.815 (0.025) & 0.908 (0.026) & 0.864 (0.023) & 0.962 (0.016) & 74.05 & 0.822 (0.046) & 0.849 (0.119) & 0.828 (0.086) \\
 & Early stopping & 0.892 (0.033) & 0.892 (0.032) & 0.960 (0.013) & 0.959 (0.015) & 100.00 & 0.000 (0.000) & 0.000 (0.000) & 0.000 (0.000) \\
 & $\ell_2$ & 0.849 (0.024) & 0.910 (0.027) & 0.903 (0.022) & 0.966 (0.017) & 73.34 & 0.809 (0.070) & 0.850 (0.160) & 0.813 (0.122) \\
 & Dropout & 0.895 (0.017) & 0.902 (0.025) & 0.939 (0.014) & 0.957 (0.018) & 76.19 & 0.831 (0.071) & 0.777 (0.224) & 0.772 (0.184) \\
\addlinespace[2pt]
\midrule
\multirow{4}{*}{0.30} & None & 0.771 (0.032) & 0.881 (0.044) & 0.820 (0.031) & 0.943 (0.036) & 70.92 & 0.819 (0.060) & 0.786 (0.194) & 0.780 (0.159) \\
 & Early stopping & 0.869 (0.044) & 0.865 (0.043) & 0.943 (0.035) & 0.941 (0.030) & 100.00 & 0.000 (0.000) & 0.000 (0.000) & 0.000 (0.000) \\
 & $\ell_2$ & 0.808 (0.033) & 0.879 (0.039) & 0.865 (0.030) & 0.946 (0.032) & 70.36 & 0.804 (0.104) & 0.777 (0.223) & 0.758 (0.188) \\
 & Dropout & 0.863 (0.026) & 0.879 (0.033) & 0.913 (0.021) & 0.941 (0.026) & 72.59 & 0.816 (0.108) & 0.740 (0.249) & 0.739 (0.225) \\
\addlinespace[2pt]
\midrule
\multirow{4}{*}{0.35} & None & 0.718 (0.038) & 0.835 (0.062) & 0.762 (0.035) & 0.906 (0.065) & 68.24  & 0.794 (0.098) & 0.718 (0.211) & 0.730 (0.186) \\
 & Early stopping & 0.827 (0.057) & 0.829 (0.057) & 0.911 (0.046) & 0.913 (0.047) & 100.00 & 0.000 (0.000) & 0.000 (0.000) & 0.000 (0.000) \\
 & $\ell_2$ & 0.752 (0.038) & 0.835 (0.058) & 0.806 (0.036) & 0.914 (0.054) & 66.33 & 0.777 (0.114) & 0.728 (0.251) & 0.710 (0.215) \\
 & Dropout & 0.817 (0.036) & 0.844 (0.041) & 0.873 (0.029) & 0.916 (0.038) & 68.19 & 0.787 (0.137) & 0.713 (0.255) & 0.709 (0.233) \\
\addlinespace[2pt]
\midrule
\multirow{4}{*}{0.40} & None & 0.658 (0.040) & 0.764 (0.079) & 0.691 (0.039) & 0.850 (0.087) & 64.83 & 0.742 (0.107) & 0.645 (0.214) & 0.661 (0.194) \\
 & Early stopping & 0.761 (0.087) & 0.755 (0.095) & 0.846 (0.091) & 0.841 (0.104) & 100.00 & 0.000 (0.000) & 0.000 (0.000) & 0.000 (0.000) \\
 & $\ell_2$ & 0.690 (0.041) & 0.773 (0.068) & 0.734 (0.044) & 0.865 (0.076) & 61.82 & 0.716 (0.134) & 0.674 (0.248) & 0.657 (0.207) \\
 & Dropout & 0.748 (0.047) & 0.783 (0.067) & 0.801 (0.044) & 0.868 (0.073) & 66.67 & 0.748 (0.158) & 0.624 (0.280) & 0.631 (0.264) \\
\addlinespace[2pt]
\midrule
\multirow{4}{*}{0.45} & None & 0.577 (0.041) & 0.638 (0.100) & 0.596 (0.043) & 0.701 (0.142) & 65.28 & 0.639 (0.151) & 0.483 (0.230) & 0.510 (0.222) \\
 & Early stopping & 0.632 (0.117) & 0.625 (0.120) & 0.689 (0.147) & 0.681 (0.154) & 100.00  & 0.005 (0.071) & 0.000 (0.002) & 0.000 (0.003) \\
 & $\ell_2$ & 0.597 (0.045) & 0.651 (0.093) & 0.624 (0.048) & 0.728 (0.138) & 61.78 & 0.637 (0.161) & 0.522 (0.273) & 0.515 (0.239) \\
 & Dropout & 0.635 (0.053) & 0.670 (0.094) & 0.670 (0.053) & 0.757 (0.122) & 63.96 & 0.616 (0.212) & 0.524 (0.257) & 0.538 (0.240) \\
\bottomrule
\end{tabular}
}
\begin{minipage}{0.98\textwidth}
\vspace{2pt}
\end{minipage}
\end{table}
\endgroup 
\subsection{Recommender System}
\label{sec:rec_experiments}
This experiment examines the performance of \eqref{prob:mainprob} in building a recommender system using the Netflix Prize dataset \cite{bennett2007netflix}. We employ a regression-based loss used in the Sparse Linear Method (SLIM) \cite{6137254} and test the recommendation performance on data without contamination, with random noise, and Bandwagon attack \cite{HMGLLX21, kapoor2017review, 10.1145/988672.988726} scenarios; we refer to \cite{zhang2023robust} for a comprehensive review of recommender system methods and data contamination settings. Our experiment considers top-\(K\) recommendation \cite{amatriain2010data}, where performance is evaluated only on the model's top-recommended items. 
The Netflix dataset contains over $100$ million ratings, from $480,189$ users on $17,770$ movies, where each user gives a rating between $1$ and $5$ for a movie they watched. We apply \(k\)-core pruning \cite{10.1145/3404835.3463245} and stratified random sampling \cite{amatriain2010data} to obtain a representable subset containing \(1{,}622\) users, \(747\) movies, and \(136{,}579\) ratings, which results in a data matrix with a density of \(11.27\%\), i.e., about eleven percent of entries are filled with nonzero ratings. 

Let \(A\in\mathbb{R}^{n\times d}\) be the user--item interaction matrix, where \(n\) is the number of users and \(d\) is the number of items. Each entry \(A_{ij}\) has a nonzero value (rating) if a user \(i\) rated an item \(j\), and otherwise has a zero. The aim of SLIM is to construct the item--item interaction matrix $X \in \mathbb{R}^{d \times d}$ by minimizing the loss $\mathcal{L}^{\rm SLIM}(X)=\frac12\| A-A X \|_F^2$. Employing the loss function and assigning a weight to each entry, we obtain the formulation,
\begin{align} \label{eq:slim-weighted-matrix}
\minimize_{ X \geq 0, \, W \geq 0} \
\frac12\| W^{1/2}\odot(A-A X)\|_F^2
+\lambda R(X)
+\gamma \sum_{i=1}^n \sum_{j=1}^d \rho(W_{ij})
\text{ subject to } X_{jj}=0,
\end{align} 
where the constraints are inherited from the SLIM method. The problem \eqref{eq:slim-weighted-matrix} can be solved by minimizing column-wise problems defined by the $j$-th column of $A$, $X$ and $W$, denoted $\h a^j$, $\x^j$ and $\w^j$ respectively, for all $j \in \{1, \dots, d\}$:
\begin{equation}
\label{eq:slim-colj-weighted}
\begin{aligned}
\minimize \limits_{\h x^j \in \mathbb{R}^d_{+},\,\h w^j \in \mathbb{R}^n_{+}} 
\tfrac12\,
  \| (\h w^j)^{1/2} 
  \odot \bigl(\h a^j - A\h x^j\bigr)\|_{2}^{2} + 
  \lambda\, R(\h x^j) + 
  \gamma\, \sum \limits_{i=1}^n \rho(w^j_i) 
\text{ subject to }
x_{j}^j=0.
\end{aligned}
\end{equation}
We propose to solve problem \eqref{eq:slim-colj-weighted} by alternately updating $\x$ using the approximation in \eqref{MM:lipgrad} and $\w$ using \eqref{MM:wc_W}. 
Note that the former approximation differs from the linear regression experiment, which employs \eqref{MM:wc}. This modification lets us solve a large-scale problem by deriving an efficient one-step update using the gradient of the current iterate, i.e., $\nabla \mathcal{L}^{\rm SLIM}(X^t)$, while avoiding an expensive inverse-matrix computation. 

We compare the SWAPE with widely practiced methods including the singular value decomposition (SVD) \cite{koren2009matrix}, User-based \(k\)-nearest neighbors (UserKNN), and Item-based \(k\)-nearest neighbors (ItemKNN) \cite{sarwar2001item}. We also implement SLIM \cite{6137254} and a modification of SLIM, where the LTS scheme is applied to the users of the data matrix, called SLIM-LTS, for comparison. 

\subsubsection{Experiment setting}
We compare the performance of all models on three data settings: Netflix data without any contamination, random noise, and bandwagon attack \cite{HMGLLX21, kapoor2017review, 10.1145/988672.988726}. Following the data model in \cite{ye2024robust}, the random noise uses a training set in which some ratings are arbitrarily swapped to a different value. We control the proportion of the flipped ratings by a parameter $\alpha \in \{ 0.1, 0.3\}$, where an increased $\alpha$ results in more noise in the data. The bandwagon attack injects artificially created, or fake, users to promote a predetermined set of target items. These users give the maximum rating to the target items, as well as items that are genuinely popular in the dataset. In our setting, each user rates $5$ for $20$ target items and the $20$ most popular items, and assigns a uniformly drawn random rating between $3$ and $5$ to $60$ arbitrarily chosen items. We use  $\alpha \in \{0.05, 0.10\}$ to control the proportion of artificial users in the training set.  

Given the data matrix $A\in\mathbb{R}^{1622\times 747}$, define a set of index pairs that captures rated elements: 
\begin{align*}
    \Omega = \{ (i,j) \in \{1, \dots, 1622\} \times \{1, \dots, 747\} \, \vert \, A_{ij} \neq 0 \}.
\end{align*}
We perform $80/20$ split at the user-level to construct the training matrix $A^{\rm tr}\in\mathbb{R}^{1622\times 747}$ such that $A^{\rm tr}$ contains approximately $80\%$ of the observations in $\Omega$, \(108{,}628\) ratings, while the rest of the elements are set to zero. The remaining \(27{,}951\) index pairs form the testing set $A^{\rm te}\in\mathbb{R}^{1622\times 747}$, resulting $A = A^{\rm tr} + A^{\rm te}$. Problem \eqref{eq:slim-weighted-matrix} formulated with $A^{\rm tr}$ is solved to obtain $(\widehat{X}, \widehat{W})$, which then makes the prediction, 
\begin{align*}
    \widehat A = A^{\mathrm{tr}}\widehat X.
\end{align*}Each element $\widehat{A}_{ij}$ can be interpreted as a score to recommend item $j$ to the user $i$. Using the scores, we describe the recommendation-making process at the user level.

For each user $i$, we identify $K$ items, e.g., $K=5$, with the highest scores among items outside the training set. The resulting set, denoted by $ \widehat{\mathcal{R}}_{\kappa}(i)$, 
contains the indices of the top $K$ movies the model recommends to each user. On the other hand, using $A^{\rm te}$, we identify items that actually received a sufficiently high rating from the user; e.g., a movie $j$ rated by user $i$ with a rating $A^{\rm te}_{ij} \geq \tau$ for some threshold $\tau > 0$. We call the latter a relevant itemset for user $i$, denoted by $\mathcal{R}(i)$. Observe that a recommendation can be considered successful if an item in $\widehat{\mathcal{R}}_{\kappa}(i)$ is contained in $\mathcal{R}(i)$. This is captured by $\text{rank}(i)$, which has a value of $1$ if the top item in $\widehat{\mathcal{R}}_{\kappa}(i)$ belongs to $\mathcal{R}(i)$, or $K$ if the $K$-th item appears in the latter set. Consequently, the system's recommendation for each user can have a rank between $1$ and $K$ if the recommendation is successful, or $\infty$ otherwise. With the above discussion, we present the evaluation metrics:
{\small
\begin{align} \label{metric:rec}
\text{\footnotesize MAE}
=
\frac{1}{27951}
\sum_{\substack{(i,j): \\ A_{ij}^{\rm te} \neq 0}}
\left|A_{ij}^{\rm te}-\widehat{ A}_{ij}\right|, \
\begin{array}{l@{}}
  \text{Hit} \\
  \text{Rate}
\end{array}
=
\frac{1}{1622}
\sum_{i=1}^{1622}
\mathbb{I}\!\left(
\widehat{\mathcal{R}}_{\kappa}(i)
\cap
\mathcal{R}(i)
\neq\varnothing
\right), \
\text{\footnotesize MRR}
=
\frac{1}{1622}
\sum_{i=1}^{1622}
\frac{1}{\text{rank}(i)}.
\end{align}}
The mean absolute error (MAE) reports the average difference between the true ratings in $A^{\rm te}$ and scores given by $\widehat{A}$. The hit rate reports the percentage of users with a successful recommendation. 
The mean reciprocal rank (MRR) further evaluates whether the order of recommendations is correct, with the highest value in the summand achieving $1$ if the top item recommendation is successful.

We conduct $30$ replications by performing independent train--test splits, and report the average values and standard deviations in the results. For each replication, our experiment evaluates the performance using different numbers of recommended items, \(K\in\{5,10,15\}\). To analyze $\widehat{W}$, we report only the percentage of retained data, omitting precision, recall, and F1 score as we assign entry-level weights in this experiment. Therefore, our model measures the proportion of entries in $A^{\rm tr}$ used in training, whereas SLIM-LTS reports the percentage of users kept during training for each columnwise problem. 

\subsubsection{Results}
We first compare all models' performance on the dataset without any contamination. In Table~\ref{tab:netfset1}, when $K=5$, we observe that our model achieves the highest hit rate and MRR of $0.481$ and $0.302$, respectively. Recall that our model solves SLIM while optimizing the weights on each entry in $A^{\rm tr}$. Using $79.1$\% of the data entries, our model improves baseline SLIM performance, e.g., from a hit rate of $0.402$. On the other hand, UserKNN and ItemKNN both show significantly lower hit rates around the value $0.015$, but outperform our model and SLIM in terms of MAE values. 
\begingroup
\setlength{\tabcolsep}{3pt}
\renewcommand{\bfseries}{\relax} 
\begin{table}[h]
\scriptsize
\caption{Performance comparison on the Netflix dataset}
\centering
\begin{adjustbox}{max width=\textwidth}
\begin{tabular}{l c ccc ccc c}
\toprule
\multirow{2}{*}{\textbf{Method}}
& \multirow{2}{*}{\textbf{MAE}}
& \multicolumn{3}{c}{\textbf{HitRate}}
& \multicolumn{3}{c}{\textbf{MRR}}
& \multirow{2}{*}{\textbf{Retained}} \\
\cmidrule(lr){3-5}\cmidrule(lr){6-8}
&
&
K=5
& K=10
& K=15
& K=5
& K=10
& K=15
& \\
\midrule
SWAPE
& 2.985 (0.007)
& \textbf{0.481 (0.019)}
& \textbf{0.696 (0.011)}
& \textbf{0.785 (0.011)}
& \textbf{0.302 (0.018)}
& \textbf{0.332 (0.018)}
& \textbf{0.339 (0.018)}
& 79.13 \\

SLIM
& 2.221 (0.007)
& 0.402 (0.010)
& 0.627 (0.012)
& 0.760 (0.010)
& 0.205 (0.006)
& 0.235 (0.005)
& 0.246 (0.005)
& -- \\

SLIM-LTS
& 1.249 (0.008)
& 0.196 (0.012)
& 0.365 (0.016)
& 0.491 (0.015)
& 0.089 (0.006)
& 0.111 (0.006)
& 0.121 (0.006)
& 80.00 \\

SVD
& 3.476 (0.005)
& 0.243 (0.009)
& 0.419 (0.012)
& 0.541 (0.013)
& 0.117 (0.006)
& 0.140 (0.006)
& 0.150 (0.006)
& -- \\

UserKNN
& 0.745 (0.003)
& 0.012 (0.003)
& 0.036 (0.005)
& 0.071 (0.008)
& 0.005 (0.001)
& 0.008 (0.001)
& 0.011 (0.001)
& -- \\

ItemKNN
& \textbf{0.729 (0.003)}
& 0.017 (0.003)
& 0.049 (0.006)
& 0.095 (0.008)
& 0.007 (0.001)
& 0.011 (0.002)
& 0.014 (0.002)
& -- \\

\bottomrule
\end{tabular}
\end{adjustbox}
\label{tab:netfset1}
\end{table}
\endgroup

The difference in performance can be analyzed by examining the magnitude of predicted scores, i.e., entries in $\widehat{A}$, produced by the models. We report that SLIM's predicted scores are generally lower than the true ratings given by users. This is attributed to a sparsity regularizer that drives the solution of the SLIM formulation toward zero, resulting in lower values in $\widehat{A}$. We observe similar behavior in other SLIM-based models and SVD, which achieve higher hit rates while mostly underpredicting the scores.
As $K$ values increase to $10$ and $15$, all models improve their hit rates and MRR, with SWAPE achieving the strongest overall recommendation performance. SLIM-LTS improves SLIM's MAE by removing $20$\% of users at the column level. However, it resulted in decreased hit rate and MRR. SVD achieves a lower hit rate than our model, while also having a higher MAE. 

In the random-noise setting, we observe a similar trend. In Table~\ref{tab:symnoise_metrics_10}, our method remains the strongest, with the highest hit rate and MRR, but with higher MAE values. Overall, all methods' evaluation metrics worsen slightly with the introduction of random noise, and performance worsens further as the contamination level $\alpha$ increases; for example, our model's hit rate decreases from $0.481$ to $0.478$ when $\alpha$ increases from $0.1$ to $0.3$ for $K=5$. 
\begingroup 
\setlength{\tabcolsep}{3pt}
\renewcommand{\bfseries}{\relax} 
\begin{table}[t]
\scriptsize
\renewcommand{\arraystretch}{1}
\caption{Performance of different models on the Netflix dataset with random noise.}
\centering
\begin{adjustbox}{max width=\textwidth}
\begin{tabular}{c l c ccc ccc c}
\toprule
\multirow{2}{*}{$\alpha$}
& \multirow{2}{*}{\textbf{Method}}
& \multirow{2}{*}{\textbf{MAE}}
& \multicolumn{3}{c}{\textbf{HitRate}}
& \multicolumn{3}{c}{\textbf{MRR}}
& \multirow{2}{*}{\textbf{Retained}} \\
\cmidrule(lr){4-6}\cmidrule(lr){7-9}
&
&
&
K=5
& K=10
& K=15
& K=5
& K=10
& K=15
& \\
\midrule

\multirow{6}{*}{0.1}
& SWAPE
& 3.027 (0.007)
& \textbf{0.481 (0.019)}
& \textbf{0.695 (0.011)}
& \textbf{0.785 (0.010)}
& \textbf{0.293 (0.020)}
& \textbf{0.322 (0.020)}
& \textbf{0.329 (0.020)}
& 79.96 \\

& SLIM
& 2.254 (0.007)
& 0.389 (0.010)
& 0.614 (0.013)
& 0.750 (0.010)
& 0.197 (0.006)
& 0.226 (0.006)
& 0.237 (0.005)
& -- \\

& SLIM-LTS
& 1.293 (0.010)
& 0.177 (0.012)
& 0.333 (0.012)
& 0.458 (0.012)
& 0.080 (0.007)
& 0.100 (0.006)
& 0.110 (0.006)
& 80.00 \\

& SVD
& 3.481 (0.005)
& 0.222 (0.011)
& 0.383 (0.014)
& 0.500 (0.013)
& 0.106 (0.006)
& 0.128 (0.006)
& 0.137 (0.006)
& -- \\

& UserKNN
& 0.779 (0.003)
& 0.006 (0.002)
& 0.018 (0.004)
& 0.041 (0.007)
& 0.003 (0.001)
& 0.004 (0.001)
& 0.006 (0.001)
& -- \\

& ItemKNN
& \textbf{0.762 (0.003)}
& 0.008 (0.002)
& 0.028 (0.005)
& 0.062 (0.009)
& 0.003 (0.001)
& 0.006 (0.001)
& 0.009 (0.001)
& -- \\

\midrule

\multirow{6}{*}{0.3}
& SWAPE
& 3.109 (0.006)
& \textbf{0.478 (0.019)}
& \textbf{0.691 (0.012)}
& \textbf{0.782 (0.010)}
& \textbf{0.289 (0.021)}
& \textbf{0.319 (0.020)}
& \textbf{0.326 (0.020)}
& 81.58 \\

& SLIM
& 2.318 (0.006)
& 0.370 (0.009)
& 0.592 (0.008)
& 0.727 (0.008)
& 0.184 (0.005)
& 0.214 (0.004)
& 0.224 (0.004)
& -- \\

& SLIM-LTS
& 1.392 (0.009)
& 0.141 (0.011)
& 0.275 (0.012)
& 0.392 (0.011)
& 0.063 (0.006)
& 0.080 (0.005)
& 0.089 (0.005)
& 80.00 \\

& SVD
& 3.487 (0.007)
& 0.183 (0.008)
& 0.323 (0.011)
& 0.433 (0.012)
& 0.090 (0.006)
& 0.109 (0.007)
& 0.117 (0.006)
& -- \\

& UserKNN
& 0.851 (0.003)
& 0.001 (0.001)
& 0.004 (0.002)
& 0.010 (0.004)
& 0.000 (0.000)
& 0.001 (0.000)
& 0.001 (0.000)
& -- \\

& ItemKNN
& \textbf{0.836 (0.003)}
& 0.002 (0.001)
& 0.007 (0.003)
& 0.023 (0.007)
& 0.001 (0.000)
& 0.001 (0.001)
& 0.003 (0.001)
& -- \\

\bottomrule
\end{tabular}
\end{adjustbox}
\label{tab:symnoise_metrics_10}
\end{table}
\endgroup

Finally, we analyze the performance of all models under the bandwagon attack. We consider $\alpha = 0.05$ and $0.10$, where the dataset includes $5$\% and $10$\% of artificially created users who try to promote target items as well as genuinely popular items, respectively. We first evaluate the models using the same metrics as before, then examine whether the bandwagon attack influences the models to recommend target items. Table~\ref{tab:orgbndwg_metrics5} shows that all models' hit rates and MRR increase as the dataset contains a higher percentage of bandwagon attackers. This indicates that the models can better identify genuinely popular items when more users (including attackers) give those items the highest rating. MAE values change only marginally across all models. 

\begingroup 
\setlength{\tabcolsep}{3pt}
\renewcommand{\bfseries}{\relax} 
\begin{table}[h]
\scriptsize
\caption{Performance of different models on the Netflix dataset under a bandwagon attack.}
\centering
\begin{adjustbox}{max width=\textwidth}
\begin{tabular}{c l c ccc ccc c}
\toprule
\multirow{2}{*}{$\alpha$}
& \multirow{2}{*}{\textbf{Method}}
& \multirow{2}{*}{\textbf{MAE}}
& \multicolumn{3}{c}{\textbf{HitRate}}
& \multicolumn{3}{c}{\textbf{MRR}}
& \multirow{2}{*}{\textbf{Retained}} \\
\cmidrule(lr){4-6}\cmidrule(lr){7-9}
&
&
&
K=5
& K=10
& K=15
& K=5
& K=10
& K=15
& \\
\midrule

\multirow{6}{*}{$0.05$}
& SWAPE
& 2.944 (0.007)
& \textbf{0.492 (0.021)}
& \textbf{0.702 (0.011)}
& 0.788 (0.011)
& \textbf{0.314 (0.017)}
& \textbf{0.343 (0.016)}
& \textbf{0.350 (0.016)}
& 77.66 \\

& SLIM
& 2.212 (0.007)
& 0.432 (0.009)
& 0.668 (0.012)
& \textbf{0.798 (0.010)}
& 0.222 (0.005)
& 0.254 (0.005)
& 0.264 (0.005)
& -- \\

& SLIM-LTS
& 1.234 (0.009)
& 0.206 (0.012)
& 0.375 (0.015)
& 0.503 (0.017)
& 0.093 (0.007)
& 0.115 (0.007)
& 0.125 (0.007)
& 80.00 \\

& SVD
& 3.534 (0.006)
& 0.244 (0.011)
& 0.423 (0.013)
& 0.543 (0.012)
& 0.119 (0.007)
& 0.142 (0.007)
& 0.152 (0.007)
& -- \\

& UserKNN
& \textbf{0.746 (0.003)}
& 0.018 (0.005)
& 0.052 (0.008)
& 0.102 (0.010)
& 0.007 (0.002)
& 0.011 (0.002)
& 0.015 (0.002)
& -- \\

& ItemKNN
& 0.747 (0.003)
& 0.012 (0.003)
& 0.037 (0.005)
& 0.074 (0.007)
& 0.004 (0.001)
& 0.008 (0.001)
& 0.010 (0.001)
& -- \\

\midrule

\multirow{6}{*}{$0.10$}
& SWAPE
& 2.901 (0.007)
& \textbf{0.506 (0.023)}
& \textbf{0.709 (0.010)}
& 0.788 (0.011)
& \textbf{0.330 (0.018)}
& \textbf{0.358 (0.017)}
& \textbf{0.364 (0.017)}
& 75.18 \\

& SLIM
& 2.214 (0.006)
& 0.445 (0.011)
& 0.682 (0.011)
& \textbf{0.810 (0.008)}
& 0.231 (0.006)
& 0.263 (0.005)
& 0.273 (0.005)
& -- \\

& SLIM-LTS
& 1.230 (0.009)
& 0.215 (0.010)
& 0.384 (0.014)
& 0.513 (0.013)
& 0.098 (0.007)
& 0.120 (0.007)
& 0.130 (0.007)
& 80.00 \\

& SVD
& 3.581 (0.006)
& 0.246 (0.011)
& 0.416 (0.010)
& 0.533 (0.011)
& 0.118 (0.005)
& 0.140 (0.005)
& 0.149 (0.005)
& -- \\

& UserKNN
& \textbf{0.748 (0.003)}
& 0.020 (0.005)
& 0.061 (0.008)
& 0.118 (0.010)
& 0.008 (0.002)
& 0.013 (0.002)
& 0.017 (0.002)
& -- \\

& ItemKNN
& 0.760 (0.003)
& 0.010 (0.002)
& 0.032 (0.005)
& 0.064 (0.007)
& 0.004 (0.001)
& 0.007 (0.001)
& 0.009 (0.001)
& -- \\

\bottomrule
\end{tabular}
\label{tab:orgbndwg_metrics5}
\end{adjustbox}
\end{table}
\endgroup

For the same bandwagon attack settings, we report all models' target hit rate and target MRR in Table~\ref{tab:target_metrics5}. In this table, the target hit rate is measured by \eqref{metric:rec} where the relevant set of the $i$-th user, $\mathcal{R}(i)$, is replaced by the set of target items. Therefore, a higher target hit rate (or MRR) indicates that the model was more influenced by the bandwagon attack. 
SWAPE achieves a target hit rate and target MRR of zero across all values of $K$ for both attack strengths, indicating that injected profiles fail to place any target item into the top-$K$ recommendations. 

\begingroup
\setlength{\tabcolsep}{3pt}
\renewcommand{\bfseries}{\relax} 
\renewcommand{\arraystretch}{1.08}

\begin{table}[t]
\scriptsize
\centering
\caption{Robustness to target promotion under bandwagon attacks. Lower Target HitRate and Target MRR indicate better robustness.}
\label{tab:target_metrics5}

\begin{tabular}{c l ccc ccc}
\toprule
\multirow{2}{*}{$\alpha$} & \multirow{2}{*}{\textbf{Method}} & \multicolumn{3}{c}{\textbf{Target HitRate}} & \multicolumn{3}{c}{\textbf{Target MRR}} \\
\cmidrule(lr){3-5}\cmidrule(lr){6-8}
 &  & $K=5$ & $K=10$ & $K=15$ & $K=5$ & $K=10$ & $K=15$ \\
\midrule
\multirow{6}{*}{$0.05$} & SWAPE & 0.0000 (0.0000) & 0.0000 (0.0000) & 0.0000 (0.0000) & 0.0000 (0.0000) & 0.0000 (0.0000) & 0.0000 (0.0000) \\
 & SLIM & 0.0366 (0.0029) & 0.0496 (0.0029) & 0.0530 (0.0045) & 0.0165 (0.0021) & 0.0183 (0.0021) & 0.0185 (0.0022) \\
 & SLIM-LTS & 0.0004 (0.0005) & 0.0016 (0.0009) & 0.0032 (0.0016) & 0.0002 (0.0003) & 0.0003 (0.0003) & 0.0004 (0.0003) \\
 & SVD & 0.0574 (0.0031) & 0.0806 (0.0039) & 0.0996 (0.0053) & 0.0284 (0.0022) & 0.0315 (0.0020) & 0.0330 (0.0021) \\
 & UserKNN & 0.0515 (0.0014) & 0.0598 (0.0022) & 0.0677 (0.0032) & 0.0265 (0.0028) & 0.0276 (0.0027) & 0.0282 (0.0027) \\
 & ItemKNN & 0.1161 (0.0166) & 0.2603 (0.0360) & 0.3955 (0.0473) & 0.0457 (0.0083) & 0.0643 (0.0099) & 0.0749 (0.0104) \\
\midrule
\multirow{6}{*}{$0.10$} & SWAPE & 0.0000 (0.0000) & 0.0000 (0.0000) & 0.0000 (0.0000) & 0.0000 (0.0000) & 0.0000 (0.0000) & 0.0000 (0.0000) \\
 & SLIM & 0.0665 (0.0038) & 0.0915 (0.0026) & 0.0954 (0.0039) & 0.0296 (0.0028) & 0.0331 (0.0027) & 0.0334 (0.0027) \\
 & SLIM-LTS & 0.0001 (0.0004) & 0.0007 (0.0009) & 0.0017 (0.0017) & 0.0000 (0.0001) & 0.0001 (0.0002) & 0.0002 (0.0002) \\
 & SVD & 0.0934 (0.0038) & 0.1245 (0.0041) & 0.1425 (0.0058) & 0.0460 (0.0027) & 0.0502 (0.0026) & 0.0516 (0.0027) \\
 & UserKNN & 0.0971 (0.0013) & 0.1074 (0.0023) & 0.1181 (0.0033) & 0.0561 (0.0025) & 0.0575 (0.0025) & 0.0583 (0.0025) \\
 & ItemKNN & 0.2525 (0.0145) & 0.4410 (0.0270) & 0.5685 (0.0314) & 0.1150 (0.0062) & 0.1397 (0.0073) & 0.1497 (0.0077) \\
\bottomrule
\end{tabular}
\end{table}
\endgroup

\section{Conclusion}
In this paper, we proposed a joint sparse sample weighting and parameter estimation framework that computes model parameters and observation weights simultaneously. The resulting bivariate formulation allows the training process to adaptively reduce the influence of unreliable observations while preserving informative data points. To solve the problem, we developed an alternating majorization-minimization algorithm and showed that, under the stated assumptions, the generated sequence satisfies non-increasing objective values and subsequence convergence, with its accumulation point being a directional stationary solution in the weakly convex case. The numerical results demonstrate the effectiveness of the SWAPE in regression, classification, and recommender system settings. In both regression and classification, the proposed methods remain competitive under overall contaminated settings. 
In recommender systems, the SWAPE achieves strong ranking performance under the no-contamination setting, random noise, and bandwagon attacks, while the target-specific results show that it limits the effects of adversarial attacks. Overall, these results suggest that our method adaptively reduces the impact of undesirable samples and achieves robust performance
when observation reliability is uncertain. 

\section*{Acknowledgments}
The third author's work was partially supported by the National Science Foundation under grant CCF-2423246.

\section*{Data Availability Statement}
We provide a Python implementation of the SWAPE at \url{https://github.com/hassan-nojavan/SWAPE}. The repository includes separate implementations for the regression, recommender system, and neural network settings considered in this work. For the neural-network experiments, we additionally provide implementations incorporating ($\ell_2$) regularization, dropout, and early stopping. 

\bibliographystyle{plain}
\bibliography{reference}

\clearpage
\appendix

\section*{Appendices}
Appendix \ref{app:proofs} provides proofs of the theoretical results, and Appendices~\ref{app:reg}--\ref{app:rec} present additional details on the three experiments described in Sections~\ref{sec:reg_experiments}--\ref{sec:rec_experiments}. We note that all experiments were executed on Southern Methodist University's HPC cluster,
featuring 128 cores powered by AMD EPYC 7763 processors. When using Gurobi, we used the default parameters unless otherwise specified. 

\section{Proofs of Results}
\label{app:proofs}

\subsection{Proof of Proposition~\ref{prop:DcWc}}
\begin{proof} 
    Combining \eqref{Lcve_Lipgrad} and the first order condition of the convex function $\mathcal{L}_{\w}^{\rm cvx} (\x)$ yields
    \begin{align*}
        &\mathcal{L}_{\w}^{\rm cvx} (\y) - \mathcal{L}_{\w}^{\rm cve} (\y) + \frac{\lipLcve}{2} \| \x - \y \|_2^2 \\ 
        &\hspace{13pc}
        \geq 
        \mathcal{L}_{\w}^{\rm cvx} (\x) - \mathcal{L}_{\w}^{\rm cve} (\x) + \mathcal{L}_{\w}^{\rm cvx} \, ' (\x; \y - \x) - \mathcal{L}_{\w}^{\rm cve} \, ' (\x; \y - \x),
    \end{align*}
    which is equivalent to 
    \begin{align*}
        \mathcal{L}_{\w}(\y) + \frac{\lipLcve}{2} \| \y \|_2^2 \geq 
        \mathcal{L}_{\w}(\x) + \frac{\lipLcve}{2} \| \x \|_2^2 + \mathcal{L}_{\w} \, ' (\x; \y - \x) + \la \lipLcve \, \x, \y - \x \ra.
    \end{align*}
\end{proof}

\subsection{Proof of Lemma~\ref{lem:d_deriv}}
\begin{proof} 
    Under \ref{assume:L_lipgrad}, the loss term is differentiable, so showing the case of \ref{assume:L_dc} is sufficient.
    The assumption \ref{assume:L_dc} implies that $\ell$ is DC and so is $\rho$ by \ref{assume:p_wc}; therefore, their directional derivatives are well-defined \cite[Lemma~3.1.5]{NLP_Bazarra_et_al}. Applying the definition of the directional derivative yields
    \begin{align*}
        F'(\h z; \h d) 
        &= \lim \limits_{\tau \downarrow 0} \displaystyle \frac{1}{\tau} \Bigg\{ \,\Big[ \, \sum \limits_{i=1}^n (w_i + \tau d^{\h w}_i) \, \ell (\h x  + \tau \h d^{\h x}; \, \h \xi^i) + \lambda \, R(\x + \tau \h d^{\h x}) + \gamma \, \sum \limits_{i=1}^n \rho (w_i  + \tau d^{\h w}_i) \, \Big] \\
        &\hspace{15pc} \hfill
        - \Big[ \, \sum \limits_{i=1}^n w_i  \, \ell (\h x ; \, \h \xi^i) + \lambda \, R(\x) + \gamma \, \sum \limits_{i=1}^n \rho (w_i ) \, \Big] \Bigg\} \\
        &= \lim \limits_{\tau \downarrow 0} \displaystyle \frac{1}{\tau} \Bigg\{ \,
        \sum \limits_{i=1}^n w_i  
        \Big[ \ell (\h x  + \tau \h d^{\h x}; \, \h \xi^i) - \ell (\h x ; \, \h \xi^i) \, \Big] 
        + \lambda \, \Big[ R(\x + \tau \h d^{\h x}) - R(\x) \Big] \\
        &\hspace{10pc}
        + \gamma \, \sum \limits_{i=1}^n
        \Big[ \, \rho (w_i  + \tau d^{\h w}_i) - \rho (w_i ) \, \Big]
        + \tau  \sum \limits_{i=1}^n
        d^{\h w}_i \,  \ell (\h x  + \tau \h d^{\h x}; \, \h \xi^i) \Bigg\} \\
        &= 
        \underbrace{
        \sum \limits_{i=1}^n \Big[ 
        w_i  \, \ell(\bullet \, ; \, \h \xi^i)'(\h x ; \h d^{\h x})
        \Big] + \lambda \, R'(\x; \h d^{\x})
        }_{F_{\w}'(\h x ; \h d^{\h x})}
        + \underbrace{\sum \limits_{i=1}^n \Big\{
        \gamma \, \rho'(w_i ; d_i^{\h w})
        + d^{\h w}_i \,  \ell (\h x ; \, \h \xi^i) 
        \Big\} }_{F_{\x}'(\h w ; \h d^{\h w})}.
    \end{align*}
\end{proof}

\subsection{Proof of Lemma~\ref{lemma:MM_DC}}
\begin{proof} 
    By convexity of $\mathcal{L}^{\rm cve}$, we obtain
    \begin{align} \label{ineq:MMdc}
        \widehat{F}_{\w^t}(\x) 
        &= \mathcal{L}_{\w^t}^{\rm cvx} (\x) - \mathcal{L}_{\w^t}^{\rm cve} (\x) + \lambda \, R(\x) \nonumber \\ 
        &\leq \mathcal{L}_{\w^t}^{\rm cvx} (\x) - \mathcal{L}_{\w^t}^{\rm cve} (\x^t) - \la \zeta_{t}, \x - \x^t \ra + \lambda \, R(\x) \text{ for any } \zeta_{t} \in \partial \mathcal{L}_{\w^t}^{\rm cve} (\x^t) \nonumber \\
        &= \widehat{G}_{\w^t}(\x; \x^t) - \frac{c_{\x}}{2} \| \x - \x^t \|_2^2,
    \end{align}
    for all $\x \in \mathcal{X}$. The conditions in \eqref{MM:cond} follow from the above.
\end{proof}

\subsection{Proof of Lemma~\ref{lemma:MM_lipgrad}}
\begin{proof}
    By \ref{assume:L_lipgrad}, we have
    \begin{align} \label{ineq:MM_lipgrad}
    \widehat{F}_{\w^t}(\x) 
    &\leq \mathcal{L}_{\w^t}(\x^t) + \la \nabla \mathcal{L}_{\w^t}(\x^t), \x - \x^t \ra + \lambda \, R(\x) + \frac{\lip}{2} \| \x - \x^t \|_2^2 \\ \nonumber
    &= \widehat{G}_{\w^t}(\x; \x^t) - \frac{c_{\x} - \lip}{2} \| \x - \x^t \|_2^2,
    \end{align}
    for any $\x \in \mathcal{X}$. Applying the condition of the lemma proves the statement.
\end{proof}

\subsection{Proof of Lemma~\ref{lemma:MM_wc_w}}
\begin{proof}  
    The first statement holds by definition. To show the second part, first note that \ref{assume:p_wc}, $P(\w)$ is weakly convex with modulus $\sigma_{\rho}$ by \ref{assume:p_wc}. Since $\mathcal{L}_{\x^{t+1}}(\w)$ is linear and we have $\mathcal{L}_{\x^{t+1}}(\w) = \mathcal{L}_{\x^{t+1}}(\bar{\w}) + \mathcal{L}_{\x^{t+1}}'(\bar{\w}; \, \w - \bar{\w})$. Combining this equality with the weakly convexity of $P$, we have
    \begin{align} \label{ineq: w_G_hat}
        \mathcal{L}_{\x^{t+1}}(\bar{\w}) + \gamma P(\bar{\w}) 
        + 
        \mathcal{L}_{\x^{t+1}}'(\bar{\w}; \w - \bar{\w}) + \gamma P'(\bar{\w}; \w - \bar{\w})
        \notag\\\leq\;\mathcal{L}_{\x^{t+1}}(\w) + \gamma P(\w)
        + \frac{c_\w}{2} \| \w - \bar{\w} \|_2^2
    \end{align}
    for any $\w, \bar{\w} \in \mathcal{W}$ since $c_{\w} \geq \gamma \sigma_{\rho}$. Choosing $\bar{\w} = \w^t$, \eqref{ineq: w_G_hat} yields the first-order inequality of a convex function. This shows that $\widehat{G}_{\x^{t+1}}(\w; \w^t)$ is convex.
\end{proof}

\subsection{Proof of Theorem~\ref{thm:conv}}
\begin{proof} 
        \begin{enumerate}[label=\textup{(\alph*)}]
        \item We will first prove for the case \ref{thm:i}. By Lemma~\ref{lemma:MM_DC} and the optimality of $\x^{t+1}$, we have 
        \begin{align*}
           \widehat{F}_{\w^t}(\x^t) = \widehat{G}_{\w^t}(\x^t; \x^t) \geq \widehat{G}_{\w^t}(\x^{t+1}; \x^t),
        \end{align*}
        where $\widehat{F}_{\w^t}(\x^t)$ is defined in \eqref{MM:F}.
        Applying \eqref{ineq:MMdc} to the above yields
        \begin{align} \label{proof:alg_a_step1}
            \widehat{F}_{\w^t}(\x^t) - \widehat{F}_{\w^t}(\x^{t+1}) \geq \frac{c_\h x}{2}\,\bigl\|\h x^{t+1}-\h x^t\bigr\|_2^{2}.
        \end{align}
        Furthermore, by Lemma~\ref{lemma:MM_wc_w} and the optimality of $\w^{t+1}$, we obtain
        \begin{align*}
            \widehat{F}_{\x^{t+1}}(\w^t) 
            = \widehat{G}_{\x^{t+1}}(\w^t; \w^t) \geq \widehat{G}_{\x^{t+1}}(\w^{t+1}; \w^t) 
            = \widehat{F}_{\x^{t+1}}(\w^{t+1}) + \frac{c_{\w}}{2} \| \w^{t+1} - \w^t \|_2^2,
        \end{align*}
        which is equivalent to 
        \begin{align} \label{proof:alg_a_step2}
            \widehat{F}_{\x^{t+1}}(\w^t)  - \widehat{F}_{\x^{t+1}}(\w^{t+1}) \geq \frac{c_\h w}{2} \| \w^{t+1}-\w^t  \|_2^{2}.
        \end{align}
        Combining \eqref{proof:alg_a_step1} and \eqref{proof:alg_a_step2} shows inequality \eqref{eq:monodec} where $\bar{c}_{\h x} = c_{\x}$. The convergence of $\{ F(\h x^t, \h w^t ) \}_{t=1}^\infty$ is given by the lower-boundedness of $F$. The proof for case \ref{thm:ii} can be shown similarly by applying \eqref{ineq:MM_lipgrad} and Lemma~\ref{lemma:MM_lipgrad}. In this case, $\bar{c}_{\h x} = c_{\x} - \lip$. The case \ref{thm:iii} can be shown by using Lemma~\ref{lemma:MM_wc}, resulting $\bar{c}_{\h x} = c_{\x}$.

       \item By summing \eqref{eq:monodec}, we have
    \begin{align} \label{proof:sum_up}
        F(\h x^{0}, \h w^{0}) - \lim \limits_{t \rightarrow \infty} F(\h x^{t}, \h w^{t})
      \geq
      \frac{\bar{c}_\x}{2} \sum \limits_{t=0}^\infty \,\bigl\|\h x^{t+1}-\h x^{t}\bigr\|_2^{2}
      +\frac{c_\h w}{2} \sum \limits_{t=0}^\infty \,\bigl\|\h w^{t+1}-\h w^{t}\bigr\|_2^{2}. 
    \end{align}
    Since $\bar c_{\h x}, c_{\h w} > 0$, we have $\| \h x^{t+1} - \h x^{t} \|_2 \rightarrow 0$ and $\| \h w^{t+1} - \h w^{t} \|_2 \rightarrow 0$.

    \item By Bolzano-Weierstrass theorem, $\{ (\h x^t,\h w^t) \}_{t=1}^\infty$ has a convergent subsequence. Next, we prove convergence to a directional stationary point for the case of (c). For any $t_j$, we have 
    \begin{align} \label{thm_c_xdesc}
    \mathcal{L}(\x^{t_j+1}, \w^{t_j}) + \lambda R(\x^{t_j+1}) + \frac{c_{\x}}{2} \| \x^{t_j+1} &- \x^{t_j} \|_2^2  
    = \widehat{G}_{\w^{t_j}}(\x^{t_j+1}; \x^{t_j}) \\
    &\hspace{-7em}\leq \widehat{G}_{\w^{t_j}}(\x; \x^{t_j})
    =  \mathcal{L}(\x, \w^{t_j}) + \lambda R(\x) + \frac{c_{\x}}{2} \| \x - \x^{t_j} \|_2^2, \nonumber
    \end{align}
    for any $\x \in \mathcal{X}$ by the optimality of $\x^{t_j+1}$. Likewise, the optimality of $\w^{t_j+1}$ yields
    \begin{align}\label{thm_c_wdesc}
    \mathcal{L}(\x^{t_j+1}, \w^{t_j+1}) + \gamma P(\w^{t_j+1}) + &\frac{c_{\w}}{2} \| \w^{t_j+1} - \w^{t_j} \|_2^2  
    = \widehat{G}_{\x^{t_j+1}}(\w^{t_j+1}; \w^{t_j}) \\
    &\hspace{-7em}\leq \widehat{G}_{\x^{t_j+1}}(\w; \w^{t_j})
    =  \mathcal{L}(\x^{t_j+1}, \w) + \gamma P(\w) + \frac{c_{\w}}{2} \| \w - \w^{t_j} \|_2^2, \nonumber
    \end{align}   
    for any $\w \in \mathcal{W}$. To combine \eqref{thm_c_xdesc} and \eqref{thm_c_wdesc}, we add $\gamma P(\w^{t_j})$ to both sides of \eqref{thm_c_xdesc}, and add $\lambda R(\x^{t_j+1})$ to both sides and choose $\w = \w^{t_j}$ in \eqref{thm_c_wdesc}. This results in
    \begin{align*}
        F(\x^{t_j+1}, \w^{t_j+1}) + \frac{c_{\x}}{2} \| \x^{t_j+1} - \x^{t_j} \|_2^2 + \frac{c_{\w}}{2} \| \w^{t_j+1} - & \w^{t_j} \|_2^2 
        \leq 
        F(\x, \w^{t_j}) + \frac{c_{\x}}{2} \| \x - \x^{t_j} \|_2^2.
    \end{align*}
    By taking the limit $t_j \rightarrow \infty$ and applying parts \ref{thm:a} and \ref{thm:b} of the theorem, we have
    \begin{align*}
        F(\x^*, \w^*) \leq F(\x, \w^*) + \frac{c_{\x}}{2} \| \x - \x^* \|_2^2, \quad \forall \x \in \mathcal{X},
    \end{align*}
    indicating that $\h x^*$ is the optimal solution of the problem $\min \limits_{\h x \in \mathcal{X}} \tilde{F}_{\h w^*}(\h x) \triangleq F(\h x, \h w^*) + \tfrac{c_\h x}{2}\|\h x-\h x^*\|_2^2$. Therefore, we have 
    \begin{align} \label{ineq:x_star}
        0 \leq \tilde{F}_{\h w^*}' (\h x^*; \h x - \h x^*) = F_{\w^*}'(\h x^*; \h x - \h x^*) \quad \forall \h x \in \mathcal{X},
    \end{align}
    where the inequality is given by the optimality of $\x^*$ and the equality holds by the definition of $\tilde{F}_{\h w^*}$. The same argument holds for $\h w^*$, and we have 
    \begin{align} \label{ineq:w_star}
    0 \leq \tilde{F}_{\h x^*}' (\h w^*; \h w - \h w^*) = F_{\x^*}'(\h w^*; \h w - \h w^*) \quad \forall \h w \in \mathcal{W},
    \end{align}
    where $\tilde{F}_{\h x^*}(\h w) \triangleq F(\h x^*, \h w) + \tfrac{c_\h w}{2}\| \h w-\h w^*\|_2^2$. Combining \eqref{ineq:x_star} and \eqref{ineq:w_star} yields 
    \begin{align*}
    0 \leq F_{\w^*}'(\h x^*; \h x - \h x^*) + F_{\x^*}'(\h w^*; \h w - \h w^*).
    \end{align*}
    By Lemma~\ref{lem:d_deriv}, the above shows that $(\x^*, \w^*)$ is a directional stationary solution of \eqref{prob:mainprob}.
\end{enumerate}
\end{proof}

\section{Regression} \label{app:reg}
We report the details of the linear regression experiment shown in Section~\ref{sec:reg_experiments}. Section~\ref{app:reg-data} discusses the data generation, including the vertical and interior clustered contamination cases. Our algorithm and subproblem details are in Section~\ref{app:reg-sub}, while the comparison methods, hyperparameter tuning process for each model, computation environment and time are reported in Section~\ref{app: reg-implementation}. 

\subsection{Data Generation} \label{app:reg-data}
Recall that \(\mathcal{I}_c \subseteq \{1,\ldots,n\}\) denotes the set of contaminated indices. Without contamination, the data matrix $A \in \mathbb{R}^{n \times d}$ and response $ \h b \in \mathbb{R}^n$ are generated following \eqref{data_generation:linear}, which is
\begin{align*}
    \h b = A\x^{ *}+\bm{\varepsilon}
    \quad \text{ where } \quad
    \bm{\varepsilon} \sim \mathcal{N}(0,I_{n}).
\end{align*}
We describe the data-generation process for interior clustered and vertical contamination below.
\begin{enumerate}[label=(\alph*)]
    \item \text{Vertical contamination}: The response vector $\h b$ is  generated using 
    \[
    \varepsilon_i \sim
    \begin{cases}
    \mathcal{N}(0,1), & i \notin \mathcal{I}_c,\\
    \mathcal{N}(\tau,1), & i \in \mathcal{I}_c,
    \end{cases}\]
    where \(\tau \in \{5,8\}\). The design matrix \(A\) remains unchanged. 
    
    \item \text{Interior clustered contamination}:
    Let \(\h c^{(1)}, \h c^{(2)} \in \mathbb{R}^d\) denote the componentwise medians of the first and the last three observations, respectively,
    \[
    c^{(1)}_j := \operatorname{median}\{A_{1j},A_{2j},A_{3j}\},
    \
    c^{(2)}_j := \operatorname{median}\{A_{n-2,j},A_{n-1,j},A_{nj}\},
    \ \forall j \in \{1,\ldots,d\}.
    \]

    For the two cloud setting, we partition \(\mathcal{I}_c\) into two equally sized disjoint subsets \(\mathcal{I}_c^{(1)}\) and \(\mathcal{I}_c^{(2)}\). The entries of $A$ are generated as  
    \[  A_{ij}\sim  
    \begin{cases}  
    \mathrm{Unif}\!\left(c_j^{(1)},\,c_j^{(1)}+0.25\right),  & i\in\mathcal{I}_c^{(1)}, \\
    \mathrm{Unif}\!\left(c_j^{(2)},\,c_j^{(2)}+0.25\right),  & i\in\mathcal{I}_c^{(2)}, \end{cases}  
    \]  
    with corresponding noise  
    \[  \epsilon_i\sim  
    \begin{cases}  
    \mathcal{N}(\tau,1),  & i\in\mathcal{I}_c^{(1)},\\  
    \mathcal{N}(-\tau,1),  & i\in\mathcal{I}_c^{(2)},  \end{cases}  
    \]
    in the response.
    The single-cloud setting is obtained by setting $\mathcal{I}_c^{(1)}=\mathcal{I}_c$ and $\mathcal{I}_c^{(2)}=\emptyset$.
\end{enumerate}

\subsection{Details of the Algorithm}
\label{app:reg-sub}
We use a modified version of the minimax concave penalty (MCP) introduced in \cite{zhang2010nearly}, which is defined in a scalar variable $w \in \mathbb{R}$ as
\begin{equation}
\label{eq:rho-mcp-reg}
\rho^{\rm mcp}(w)
=
\begin{cases}
\displaystyle
\theta |w-1|
-\frac{(w-1)^2}{2\eta},
& |w-1|\leq \eta\theta,\\[8pt]
\displaystyle
\frac{\eta\theta^2}{2},
& |w-1|>\eta\theta.
\end{cases}
\end{equation}
The MCP has two parameters \(\theta>0\) and \(\eta>1\).
Choosing the MCP as $\rho$ in \eqref{prob:mainreg}, the proposed formulation  is written as
\begin{equation}
\label{eq:reg-objective-mcp}
F^{\rm reg}(\h x,\h w)
=
\frac{1}{2}
\left\|
\h w^{1/2}\odot(\h b-A\h x)
\right\|_2^2
+
\gamma\sum_{i=1}^{n}\rho^{\rm mcp}(w_i).
\end{equation}
To solve \eqref{eq:reg-objective-mcp}, we apply our algorithm \eqref{alg:subproblems} while 
using surrogates \eqref{MM:wc} and \eqref{MM:wc_W} for $\x$ and $\w$ subproblems. The algorithm is summarized in Algorithm~\ref{alg:awcs-lasso-sum}.
\begin{algorithm}[H]
\small
\caption{Proposed Method for Solving \eqref{eq:reg-objective-mcp}}
\label{alg:awcs-lasso-sum}
\begin{algorithmic}[1]
  \REQUIRE \(A\in\mathbb{R}^{n\times d},\, \h b\in\mathbb{R}^{n},\, \gamma>0,\,
  c_\h x>0,\, c_\h w>0,\, \varepsilon>0,\, T_{\max}\)
  \STATE Initialize \(\x^{0}\in\mathbb{R}^{d}\), \(\w^{0}\in\mathbb{R}^{n}_+\) 
    \STATE \(F^{\rm reg,0}= F^{\rm reg}\bigl(\h x^{0},\h w^{0}\bigr),\;t\gets0\)  \WHILE{\(\tfrac{\lvert F^{\rm reg,t}-F^{\rm reg,t-1}\rvert}{\lvert F^{\rm reg,t}\rvert}>\varepsilon\) and \(t<T_{\max}\)}
    \vspace{3pt}
    \STATE Solve the following subproblems:
    \begin{subequations}\label{eq:awcs-lasso-sum-subprobs}
    \begin{align}
      \h x^{t+1}
      &= \argmin_{\h x\in\mathbb{R}^{d}}
      \Bigl[
        \tfrac12 \| (\h w^{t})^{1/2}\odot (\h b-A\h x)\|_2^2
        +\tfrac{c_\h x}{2}\|\h x-\h x^{t}\|_2^2
      \Bigr],
      \label{eq:awcs-lasso-sum-xsub}\\[6pt]
      \h w^{t+1}
      &= \argmin_{\substack{\h w\in\mathbb{R}_+^{n}}}
      \Bigl[
        \tfrac12 \| \h w^{1/2}\odot (\h b-A\h x^{t+1})\|_2^2
        +\gamma\sum_{i=1}^{n}\rho^{\rm mcp}(w_i)
        +\tfrac{c_\h w}{2}\|\h w-\h w^{t}\|_2^2
      \Bigr],
      \label{eq:awcs-lasso-sum-wsub}
    \end{align}
    \end{subequations}
    \STATE \(F^{\rm reg,t+1}\gets F^{\rm reg}\bigl(\h x^{t+1},\h w^{t+1}\bigr)\)
    \STATE \(t\gets t+1\)
  \ENDWHILE
  \STATE \textbf{return} \((\h x^{t},\h w^{t})\)
\end{algorithmic}
\end{algorithm}

\subsubsection{Closed-Form Solution for $\x$-subproblem \eqref{eq:awcs-lasso-sum-xsub}}
The $\x$ and $\w$ subproblems of Algorithm~\ref{alg:awcs-lasso-sum} admit closed-form solutions. We first derive a closed-form solution for the \(\h x\)-subproblem given in \eqref{eq:awcs-lasso-sum-xsub}.
For fixed observation weights \(\h w^t\), the regression \(\h x\)-subproblem in \eqref{eq:awcs-lasso-sum-xsub} reduces to a proximal weighted least-squares problem. Let \(W^t\in\mathbb{R}^{n\times n}\) denote the diagonal matrix whose
diagonal entries are given by \(\h w^t\). The objective in \eqref{eq:awcs-lasso-sum-xsub} is a convex quadratic function of \(\h x\). Hence, its minimizer is obtained by the first-order optimality condition
\[
A^{\mathsf T}W^t(A\h x^{t+1}-\h b)
+
c_{\h x}(\h x^{t+1}-\h x^t)
=0,
\]
therefore, the \(\h x\)-update admits the closed-form representation,
\[
\h x^{t+1}
=
\left(A^{\mathsf T}W^tA+c_{\h x}I\right)^{-1}
\left(A^{\mathsf T}W^t\h b+c_{\h x}\h x^t\right).
\]

\subsubsection{Closed-Form Solution for $\w$-subproblem \eqref{eq:awcs-lasso-sum-wsub}}
Here, we derive a closed-form update for the \(\h w\)-subproblem in \eqref{eq:awcs-lasso-sum-wsub}. Observe that \eqref{eq:awcs-lasso-sum-wsub} can be decomposed into univariate problems for each element of $\w$.
For an arbitrary coordinate \(i\in\{1,\ldots,n\}\), define
\begin{equation*}
r_i^{t+1}
=
\bigl((\h b-A\h x^{t+1})_i\bigr)^2 \quad
\text{ and } \quad
t
=
w_i^t-\frac{r_i^{t+1}}{2c_{\h w}}.
\end{equation*}
By substituting the above and removing constant terms, we obtain
the following problem for the \(i\)-th update of \(\h w\) in \eqref{eq:awcs-lasso-sum-wsub}:
\begin{equation}
\label{eq:appendix-scalar-subproblem}
\min_{w\ge 0}\;
\frac{c_{\h w}}{2}(w-t)^2
+\gamma\rho^{\rm mcp}(w),
\end{equation}
where the notation $w$ replaces $w_i$. We verify that the $\rho^{\rm mcp}$ defined in \eqref{eq:rho-mcp-reg} is $\frac{1}{\eta}$-weakly convex. Consequently, when
\begin{equation} \label{eq:curvature-cond}
 c_\h w > \frac{\gamma}{\eta},
\end{equation}
the problem \eqref{eq:appendix-scalar-subproblem} is strongly convex.

Under condition \eqref{eq:curvature-cond}, the unique minimizer of \eqref{eq:appendix-scalar-subproblem} can be characterized by partitioning the feasible set according to the breakpoints of \(\rho^{\rm mcp}\). From \eqref{eq:rho-mcp-reg}, the transition between the quadratic and constant branches occurs when
\(
|w-1|=\eta\theta,
\)
giving the breakpoints \(1-\eta\theta\) and \(1+\eta\theta\). In addition, \(w=1\) is a point of nondifferentiability due to the term \(|w-1|\). Hence, the relevant breakpoints are
\(
1-\eta\theta,\ 1,\text{ and } 1+\eta\theta.
\) If \(\eta\theta<1\), all three
breakpoints lie in the feasible region \(w\ge0\). If \(\eta\theta\ge1\), then \(1-\eta\theta\le0\), and the lower constant branch does not occur in the interior of the feasible set. Accordingly, the minimizer is obtained by examining the first-order conditions on each smooth region together with the boundary and nondifferentiable points \(w=0\) and \(w=1\).

\noindent\textbf{Zone (0): \(w=0\).}

Since \(\phi\) is strongly convex under \eqref{eq:curvature-cond}, the boundary point \(w=0\) is optimal whenever
\(
\phi'(0^+)\ge0.
\)
Indeed, in this case the objective cannot decrease as \(w\) moves from the boundary into the feasible region \(w>0\). At \(w=0\) the applicable branch of \(\rho^{\rm mcp}\) depends on the value of \(\eta\theta\).
\begin{itemize}
\item If \(\eta\theta\ge1\), then \(w=0\) lies in the quadratic branch.
Hence,
\(
\phi'(0^+)
=
-c_{\h w}t
+
\gamma\left(\frac{1}{\eta}-\theta\right).
\)
Therefore, if
\(
t
\le
\frac{\gamma}{c_{\h w}}
\left(\frac{1}{\eta}-\theta\right),
\)
then \(\phi'(0^+)\ge0\), and consequently
\(
w^*=0.
\)

\item If \(\eta\theta<1\), then \(w=0\) lies in the constant branch, and
\(
\phi'(0^+)=-c_{\h w}t.
\)
Thus, if \(t\le0\), then \(\phi'(0^+)\ge0\), and consequently
\(
w^*=0.
\)
\end{itemize}
\noindent\textbf{Zone (I): \(0<w\le1-\eta\theta\).}

This zone exists only when \(\eta\theta<1\). On this interval,
\(\rho^{\rm mcp}\) is constant.
Therefore,
\(
\phi'(w)=c_{\h w}(w-t),
\)
and the stationary point is \(w=t\). If
\(
0<t\le1-\eta\theta,
\)
then this stationary point lies in Zone (I), and strong convexity of
\(\phi\) implies that
\(
w^*=t.
\)
If \(t\le0\), then the stationary point lies outside the feasible set and,
as shown in Zone (0), \(\phi'(0^+)\ge0\), so the minimizer is \(w^*=0\).
If \(t>1-\eta\theta\), then
\(
\phi'(w)<0
\
\text{for all }0<w\le1-\eta\theta,
\) so \(\phi\) is decreasing throughout Zone (I). Hence, the minimizer cannot
occur in Zone (I), and we proceed to Zone (II).

\noindent\textbf{Zone (II): \(1-\eta\theta<w<1\).}

On this interval, the first-order condition is
\(
\phi'(w)
=
c_{\h w}(w-t)
-\gamma\theta
+\frac{\gamma}{\eta}(1-w)
=0,
\)
which yields the stationary point
\[
w^*
=
1-
\frac{c_{\h w}(1-t)-\gamma\theta}
{c_{\h w}-\gamma/\eta}.
\]
\begin{itemize}
    \item  If \(\eta\theta<1\), Zone (II) is bounded below by
\(1-\eta\theta>0\). The condition
\(
1-\eta\theta<w^*<1
\)
is equivalent to
\(
1-\eta\theta
<
t
<
1-\frac{\gamma\theta}{c_{\h w}}.
\)
Hence, for \(t\) in this interval, the stationary point lies in Zone (II)
and, by strong convexity, is the unique minimizer.

 \item If \(\eta\theta\ge1\), the lower breakpoint satisfies
\(1-\eta\theta\le0\), so the feasible portion of Zone (II) is
\(0<w<1\). In this case, the condition
\(
0<w^*<1
\)
is equivalent to
\(
\frac{\gamma}{c_{\h w}}
\left(\frac{1}{\eta}-\theta\right)
<
t
<
1-\frac{\gamma\theta}{c_{\h w}}.
\)
At the upper endpoint, \(t=1-\gamma\theta/c_{\h w}\), the stationary point reaches \(w=1\), which is considered in Zone (III).
\end{itemize}

\noindent\textbf{Zone (III): \(w=1\).}

The point \(w=1\) is nondifferentiable because of the term \(|w-1|\). Therefore, optimality is characterized by
\(
0\in\partial\phi(1),
\)
or, equivalently,
\(
\phi'(1^-)\le0\le\phi'(1^+).
\)
The one-sided derivatives are
\begin{align*}
\phi'(1^-)
&=
c_{\h w}(1-t)-\gamma\theta,
\qquad
\phi'(1^+)
=
c_{\h w}(1-t)+\gamma\theta.
\end{align*}
Hence, \(w=1\) is optimal when
\(
1-\frac{\gamma\theta}{c_{\h w}}
\le t\le
1+\frac{\gamma\theta}{c_{\h w}}.
\)
If
\(
t>1+\gamma\theta/c_{\h w},
\)
then \(\phi'(1^+)<0\), so the objective continues to decrease to the right of \(w=1\), and we proceed to Zone (IV).

\noindent\textbf{Zone (IV): \(1<w<1+\eta\theta\).}

On this interval, the first-order condition becomes
\(
\phi'(w)
=
c_{\h w}(w-t)
+\gamma\theta
-\frac{\gamma}{\eta}(w-1)
=0,
\)
which yields
\(
w^*
=
1+
\frac{c_{\h w}(t-1)-\gamma\theta}
{c_{\h w}-\gamma/\eta}.
\)
For this stationary point to lie in Zone (IV), it must satisfy
\(
1<w^*<1+\eta\theta.
\)
Since \(c_{\h w}-\gamma/\eta>0\) under \eqref{eq:curvature-cond}, these inequalities are equivalent to
\(
1+\frac{\gamma\theta}{c_{\h w}}
<t
<
1+\eta\theta.
\)

If \(t\ge1+\eta\theta\), the stationary point is at or beyond the upper breakpoint \(1+\eta\theta\), hence the minimizer cannot lie in the interior of Zone (IV), and we proceed to Zone (V). 
\noindent\textbf{Zone (V): \(w\ge1+\eta\theta\).} On this interval,
\(
\phi'(w)
=
c_{\h w}(w-t).
\)
Thus, the stationary point is \(w=t\), which lies in Zone (V) when
\(
t\ge1+\eta\theta.
\)

Combining the above cases, under condition \eqref{eq:curvature-cond}, the unique minimizer of \eqref{eq:appendix-scalar-subproblem} is

\begin{equation}
\label{eq:mcp-closed-form}
w_i^{t+1} =
\begin{cases}
0,
& \eta\theta<1 \ \text{and}\ t\le 0,
\\[8pt]
0,
& \eta\theta\ge 1 \ \text{and}\ t\le \dfrac{\gamma}{c_\h w}\!\left(\dfrac{1}{\eta}-\theta\right),
\\[10pt]
t,
& \eta\theta<1\ \text{and}\ 0< t \le 1-\eta\theta,
\\[10pt]
\displaystyle
1-\frac{c_\h w(1-t)-\gamma\theta}{c_\h w-\gamma/\eta},
& \eta\theta<1 \ \text{and}\ 1-\eta\theta<t<1-\dfrac{\gamma\theta}{c_\h w},
\\[14pt]
\displaystyle
1-\frac{c_\h w(1-t)-\gamma\theta}{c_\h w-\gamma/\eta},
& \eta\theta\ge 1 \ \text{and}\ \dfrac{\gamma}{c_\h w}\!\left(\dfrac{1}{\eta}-\theta\right)<t<1-\dfrac{\gamma\theta}{c_\h w},
\\[14pt]
1,
& 1-\tfrac{\gamma\theta}{c_\h w}\ \le\ t\ \le\ 1+\tfrac{\gamma\theta}{c_\h w},
\\[10pt]
\displaystyle
1+\frac{c_\h w(t-1)-\gamma\theta}{\,c_\h w-\gamma/\eta\,},
& 1+\tfrac{\gamma\theta}{c_\h w}\ <\ t\ <\ 1+\eta\theta,
\\[14pt]
t,
& t \ge 1+\eta\theta.
\end{cases}
\end{equation}

\subsection{Computational Details}
\label{app: reg-implementation}
\subsubsection{Comparison Methods}
\label{app:reg-comparison}
We consider five benchmark methods in Section~\ref{sec:reg_experiments}: least squares, least absolute deviation (LAD), Huber regression \cite{Huber1964}, least trimmed squares (LTS) \cite{Rousseeuw1984, Rousseeuw1987}, and \(\Theta\)-IPOD \cite{She2011}. Given a data matrix $A \in \mathbb{R}^{n \times d}$ and response vector $\h b \in \mathbb{R}^{n}$, 
\begin{itemize}
    \item Least squares solves the problem
    \begin{equation*}
        \min_{\h x \in \mathbb{R}^d} \;\; \|\h b-A\h x\|_2^2;
    \end{equation*}

    \item LAD solves 
    \begin{equation*}
    \min_{{\mathbf x}\in\mathbb{R}^d} \|{\mathbf b}  - A{\mathbf x}\|_1;
    \end{equation*}

    \item Given parameters $\delta, \alpha_H > 0$, the Huber regression solves
    \begin{equation*}
        \min_{\h x\in\mathbb{R}^d}
\sum_{i=1}^{n}\ell_{\delta}\!\left( r_i (\x) \right)
+\frac{\alpha_H}{2}\|\h x\|_2^2
\qquad \text{where} \qquad
\ell_{\delta}(r)
=
\begin{cases}
\dfrac{1}{2}r^2, & |r|\le \delta,\\[3pt]
\delta |r|-\dfrac{1}{2}\delta^2, & |r|>\delta
\end{cases}
    \end{equation*}
\end{itemize}
and $r_i (\x) \triangleq (\h b - A \x)_i$ denotes the residual of the $i$-th observation.
We note that the above methods use all observations in $A$ and $\h b$. The LTS and \(\Theta\)-IPOD are designed to handle outliers in the data. We summarize the two methods below.
\begin{itemize}
    \item LTS sorts the residuals given by a model $\x$ such that $r_{[1]}^2 (\x) \leq r_{[i]}^2  (\x) \leq r_{[n]}^2 (\x)$. For a given positive scalar $h < n$, LTS solves the problem 
    \begin{equation*}
        \min_{\h x\in\mathbb{R}^d}
\sum_{i=1}^{h} r_{[i]}^2(\h x).
    \end{equation*}
    We compute the solution of LTS using the FAST-LTS algorithm \cite{rousseeuw2006computing}, which iteratively chooses subsets of \(h\) observations.
    \item Finally, the hard-thresholding version of \(\Theta\)-IPOD, which introduces an observation-specific mean-shift parameter $\nu_i$ to model vertical outliers, solves 
    \begin{equation*}
\min_{\h x,\,\h \nu}
\quad
\frac{1}{2}\|A\h x+\h \nu-\h b\|_2^2
+
\sum_{i=1}^n p_{\tilde{\theta}}(\nu_i),
\qquad
p_{\tilde{\theta}}(\nu_i)
=
    \begin{cases}
\tilde{\theta}|\nu_i|-\dfrac{\nu_i^2}{2},
& |\nu_i|<\tilde{\theta},\\[4pt]
\dfrac{\tilde{\theta}^2}{2},
& |\nu_i|\geq\tilde{\theta}.
    \end{cases}
    \end{equation*}
    In our experiments, observation \(i\) is classified as an outlier whenever \(\widehat{\nu}_i\neq 0, \) and as a non outlier otherwise. This rule is used to compute the corresponding detection precision, recall, and F1-score. Accordingly, we define the set of detected outliers as 
    $\widehat{\mathcal I}_c \triangleq \left\{ i\in\{1,\ldots,n\} : \widehat{\nu}_i\neq0 \right\}$.  
    The proportion of observations retained is  
    $\mathrm{Retained} = \frac{100}{n} \left| \left\{ i\in\{1,\ldots,n\} : \widehat{\nu}_i=0 \right\} \right|$. 
\end{itemize}

\subsubsection{Hyperparameter Selection}
\label{app:reg-tuning}
All tuning parameters are selected independently for each replication, using the leave-one-out cross validation (LOOCV) method.
For our method, LOOCV selects the regularization parameter \(\gamma\) that yields the minimum median MSE, while the MCP parameters are fixed at \(\theta=1\) and \(\eta=3\). Following \cite{tay2023elastic}, we generate \(L=300\) candidate values for $\gamma$ using a
logarithmic scale,
\[
\gamma_j
=
\gamma_{\max}
\left(
\frac{\gamma_{\min}}{\gamma_{\max}}
\right)^{\frac{j-1}{L-1}},
\qquad j=1,\ldots,L,
\]
where $\gamma_{\min}$ and $\gamma_{\max}$ are prespecified bounds.

The benchmark methods least squares and LAD do not require tuning. We use the library \texttt{scikit-learn}'s default parameters \(\delta=1.35\) and \(\alpha_H=10^{-4}\) for Huber regression. Following the suggestion in \cite{vcivzek2000least}, we use $h = \lfloor 0.75\, n \rfloor$ for LTS. 
For \(\Theta\)-IPOD, we tune the threshold parameter \(\tilde{\theta}\) over the fixed grid $
\tilde{\theta} \in \{1.5,\,2.0,\,2.5,\,3.0,\,3.5\}$ using the LOOCV median MSE criterion.

\subsubsection{Implementation}

Experiments were implemented in Python using \texttt{NumPy} \cite{harris2020array} and \texttt{scikit-learn}.
Table~\ref{tab:reg_runtime} reports the training wall-clock time for a representative vertical-contamination setting in Section~\ref{sec:reg-results}, with \(\alpha=0.3\) and \(\tau=5\). 

\begin{table}[!ht]
\centering
\scriptsize
\caption{Training wall-clock time (seconds) for the vertical-contamination setting with $\alpha=0.3$ and $\tau=5$. Reported values are the mean (and standard deviation) over 200 independent
replications.}
\label{tab:reg_runtime}
\begin{tabular}{lc}
\toprule
Method & Training Time (s) \\
\midrule
SWAPE & 0.0025 (0.0009) \\
$\Theta$-IPOD & 0.0045 (0.0049) \\
Least Squares & 0.0002 (0.0001) \\
LAD & 0.7347 (0.0144) \\
Huber & 0.0116 (0.0012) \\
LTS & 0.0069 (0.0011) \\
\bottomrule
\end{tabular}
\end{table}

\section{Binary Classification} \label{app:binary}
We report the details of the binary classification experiment in Section~\ref{sec:nn_experiments}. Following the neural network structure and resulting formulation in Section~\ref{app:nn_derivation}, our algorithm and its subproblem derivations are presented in Section~\ref{app:nn_subprob_deriv}. The details of comparison methods, hyperparameter tuning process, computation environment and time are reported in Section~\ref{app:nn-comp-details}.

\subsection{Details of the Neural Network in the Binary Classification Formulation} \label{app:nn_derivation}
We derive the problem \eqref{prob:mainnn} using a fully-connected feed-forward neural network.
Let \(A\in\mathbb{R}^{n\times q}\) denote the feature matrix, with \(\h a_i\in\mathbb{R}^q\) representing the feature vector of observation \(i \in \{ 1, \dots, n\}\), and let \({\h b}\in\{0,1\}^n\) the vector of observed binary labels, and \(\w\in\mathbb{R}_+^n\) the observation weight vector. We consider a neural network with two hidden layers of widths \(m_1\) and \(m_2\). Let \(\x=\bigl(V_1,\h s_1,V_2,\h s_2,\h v_3,s_3\bigr)\in\mathbb{R}^{d}\) collect all network parameters, where \( V_1\in\mathbb{R}^{m_1\times q}, V_2\in\mathbb{R}^{m_2\times m_1},\text{ and }\h v_3\in\mathbb{R}^{m_2},\) are the network weights, while \(\h s_1\in\mathbb{R}^{m_1}, \;\h s_2\in\mathbb{R}^{m_2}, \;s_3\in\mathbb{R},\) are the corresponding bias parameters, and \(d=m_1q+m_1m_2+m_1+2m_2+1\). For observation \(i\), the hidden-layer representations are
\[
\h h_i^{(1)}
=
\phi_\beta(V_1\h a_i+\h s_1),
\qquad
\h h_i^{(2)}
=
\phi_\beta(V_2\h h_i^{(1)}+\h s_2),
\] where \(\phi_\beta(u) =  \beta^{-1}\log(1+e^{\beta u})\) is the softplus activation function, applied componentwise to vector inputs. We consider the weighted binary cross-entropy loss
\begin{align} \label{app:nn_loss_term}
\mathcal{L}(\x,\w)
=
-\sum_{i=1}^{n} w_i
\left[
 b_i\log   p_i(\x)
+
(1- b_i)
\log\!\left(1-  p_i(\x)\right)
\right].
\end{align}
where \( p_i(\x)=\sigma\!\left(\h v_3^{\mathsf T}\h h_i^{(2)}+s_3\right)\) is the fitted class-one probability, and \(\sigma(u)=(1+e^{-u})^{-1}\) is the logistic function. By applying loss \eqref{app:nn_loss_term} to our main problem \eqref{prob:mainprob}, we obtain the training problem \eqref{prob:mainnn}, which is
\begin{align} \label{app:nn_prob}
\minimize_{\x\in\mathbb{R}^d,\;\w\in\mathbb{R}_+^n}\
F^{\rm NN}(\x,\w)
={}&
-\sum_{i=1}^{n} w_i
\left[
 b_i\log  p_i(\x)
+
(1- b_i)
\log\!\left(1- p_i(\x)\right)
\right]\notag\\
&\hspace{11em}
+\lambda R(\x)
+\gamma\sum_{i=1}^{n}\rho^{\rm lasso}(w_i).
\end{align}
Here, we use a horizontally shifted 
LASSO penalty defined in a scalar variable $w \in \mathbb{R}$ as 
\begin{align} \label{def:shifted_lasso} 
    \rho^{\rm lasso}(w) =|1-w|.
\end{align}
We implement the problem \eqref{app:nn_prob} without regularization on $\x$, i.e., $R(\x) = 0$, and with $\ell_2$-regularization on the network's edge weights, $R(\x) = \frac{1}{2} \left(\|V_1\|_F^2 + \|V_2\|_F^2 + \|\h v_3\|_2^2\right)$.

\subsection{Details of the Algorithm} \label{app:nn_subprob_deriv}
We apply the proposed method \eqref{alg:subproblems} to solve \eqref{app:nn_prob} using surrogates \eqref{MM:wc} and \eqref{MM:wc_W} for $\x$ and $\w$ subproblems, respectively. The resulting algorithm is summarized in Algorithm~\ref{alg:awcs-nn}. In this experiment, we use Adam \cite{kingma2015adam} to solve the $\x$-subproblem \eqref{eq:awcs-nn-xsub}, therefore, we only present the closed-form solution derivation for the $\w$-subproblem.
\begin{algorithm}[H]
\small
\caption{Proposed Method for Solving \eqref{app:nn_prob}}
\label{alg:awcs-nn}
\begin{algorithmic}[1]
  \REQUIRE \(\{\h\xi^i=(\h a_i, b_i)\}_{i=1}^{n},\,
  \lambda\ge0,\,\gamma>0,\,c_\h x>0,\,c_\h w>0,\,
  \varepsilon>0,\,T_{\max}\)
  \STATE Initialize \(\h x^{0}\in\mathbb{R}^d\) and
  \(\h w^{0}\in\mathbb{R}_+^{n}\)
  \STATE \(F^{{\rm NN},0}
  =F^{\rm NN}\bigl(\h x^{0},\h w^{0}\bigr), \; t\gets0\)

  \WHILE{\(\tfrac{
      \lvert F^{{\rm NN},t}-F^{{\rm NN},t-1}\rvert
    }{
      \lvert F^{{\rm NN},t}\rvert\
    }>\varepsilon\) and \(t<T_{\max}\)}
    \vspace{3pt}
    \STATE Solve the following subproblems:
    \begin{subequations}
    \label{eq:awcs-nn-subprobs}
    \begin{align}
    \h x^{t+1}
    &=
    \argmin_{\h x\in\mathbb{R}^d}
    \Biggl[
    -\sum_{i=1}^{n}w_i^t
    \left[
    b_i\log p_i(\h x)
    +
    (1-b_i)\log\!\left(1-p_i(\h x)\right)
    \right]\notag\\
&\hspace{12em}
    +\lambda R(\h x)
    +\frac{c_\h x}{2}
    \|\h x-\h x^t\|_2^2
    \Biggr],
    \label{eq:awcs-nn-xsub}
    \\[6pt]
    \h w^{t+1}
    &=
    \argmin_{\h w\in\mathbb{R}_+^{n}}
    \Biggl[
    -\sum_{i=1}^{n}w_i
    \left[
    b_i\log p_i(\h x^{t+1})
    +
    (1-b_i)\log\!\left(1-p_i(\h x^{t+1})\right)
    \right]\notag\\
    &\hspace{12em}
    +\gamma\sum_{i=1}^{n}\rho^{\rm lasso}(w_i)
    +\frac{c_\h w}{2}
    \|\h w-\h w^t\|_2^2
    \Biggr].
    \label{eq:awcs-nn-wsub}
    \end{align}
    \end{subequations}

    \STATE \(F^{{\rm NN},t+1}
    \gets F^{\rm NN}\bigl(\h x^{t+1},\h w^{t+1}\bigr)\)
    \STATE \(t\gets t+1\)
  \ENDWHILE

  \STATE \textbf{return} \((\h x^{t},\h w^{t})\)
\end{algorithmic}
\end{algorithm}

\subsubsection{Closed-Form Solution for $\w$-subproblem \eqref{eq:awcs-nn-wsub}}
Here, we derive a closed-form solution for the  $\w$-subproblem.
For fixed \(\h x^{t+1}\), define the loss associated with observation \(i\) by $\label{eq:nn-loss-fixed}
\ell_i^{t+1}
=
-
\left[
b_i\log p_i(\h x^{t+1})
+
(1-b_i)\log\!\left(1-p_i(\h x^{t+1})\right) \right]$.
The \(\h w\)-subproblem in \eqref{eq:awcs-nn-wsub} is separable across observations. For each \(i\in\{1,\ldots,n\}\), the update of \(w_i\) reduces
to a univariate problem
\begin{equation}
\label{eq:nn-w-scalar}
\min_{w\geq0} \
w\ell_i^{t+1}
+
\gamma|w-1|
+
\frac{c_\h w}{2}(w-w_i^t)^2,
\end{equation}
where the LASSO penalty in \eqref{def:shifted_lasso} is used.
Defining $z
=
w_i^t-\frac{\ell_i^{t+1}}{c_{\h w}}$ and 
completing the square, problem \eqref{eq:nn-w-scalar} reduces to solving a univariate problem
\begin{equation}
\label{eq:nn-lasso-reduced}
\min_{w\geq0} \, \phi(w) =
\frac{c_{\h w}}{2}(w-z)^2
+
\gamma|w-1|.
\end{equation}
The above problem is strongly convex and has a unique global minimizer. To derive a closed-form solution of \eqref{eq:nn-lasso-reduced}, we consider several intervals of the variable $w$ consisting of the end point \(w=0\) and the nondifferentiable point \(w=1\).

\medskip
\noindent\textbf{Zone (0): \(w=0\).} At the boundary \(w=0\), the right derivative is
\(
\phi'(0^+)
=
-c_{\h w}z-\gamma.
\)
The boundary point is optimal whenever
\(\phi'(0^+)\geq0\), which is equivalent to
\(
z\leq-\frac{\gamma}{c_{\h w}}.
\)

\medskip
\noindent\textbf{Zone (I): \(0<w<1\).} On this interval,
\(
|w-1|=1-w
\),
and hence
\(
\phi'(w)
=
c_{\h w}(w-z)-\gamma.
\)
The first-order condition \(\phi'(w)=0\) gives
\(
w
=
z+\frac{\gamma}{c_{\h w}}.
\)
For this stationary point to lie in \(0<w<1\), we require
\(
-\frac{\gamma}{c_{\h w}}<z<1-\frac{\gamma}{c_{\h w}}.
\)

\medskip
\noindent\textbf{Zone (II): \(w=1\).} The point \(w=1\) is nondifferentiable because of the term \(|w-1|\). Optimality is characterized by
\(
0\in\partial\phi(1),
\)
The one-sided derivatives are
\(
\phi'(1^-)
=
c_{\h w}(1-z)-\gamma,
\ \text{and }
\phi'(1^+)
=
c_{\h w}(1-z)+\gamma.
\)
Thus, \(w=1\) is optimal if and only if
\(
1-\frac{\gamma}{c_{\h w}}
\leq
z
\leq
1+\frac{\gamma}{c_{\h w}}.
\)

\medskip
\noindent\textbf{Zone (III): \(w>1\).}

On this interval,
\(
|w-1|=w-1
\),
so
\(
\phi'(w)
=
c_{\h w}(w-z)+\gamma.
\)
The first-order condition gives
\(
w
=
z-\frac{\gamma}{c_{\h w}}.
\)
For this stationary point to satisfy \(w>1\), we require
\(
z>1+\frac{\gamma}{c_{\h w}}.
\)

Combining the above cases, the unique minimizer of
\eqref{eq:nn-lasso-reduced} is
\begin{equation}
\label{eq:nn-w-closed-form}
w_i^{t+1}
=
\begin{cases}
0,
& z\leq-\dfrac{\gamma}{c_{\h w}},
\\[8pt]
\displaystyle
z+\dfrac{\gamma}{c_{\h w}},
&
-\dfrac{\gamma}{c_{\h w}}
<
z
<
1-\dfrac{\gamma}{c_{\h w}},
\\[10pt]
1,
&
1-\dfrac{\gamma}{c_{\h w}}
\leq
z
\leq
1+\dfrac{\gamma}{c_{\h w}},
\\[10pt]
\displaystyle
z-\dfrac{\gamma}{c_{\h w}},
&
z
>
1+\dfrac{\gamma}{c_{\h w}}.
\end{cases}
\end{equation}

\subsection{Computational Details} \label{app:nn-comp-details}

\subsubsection{Comparison Methods}
\label{app:nn-comparison}

Our experiment compares the proposed formulation \eqref{prob:mainnn} to a baseline method which solves the same problem only for $\x$ while fixing all weights to 1, i.e., $\w = \bm{1}_n$. For both models, we implement four variants: no regularization, $\ell_2$-regularization, dropout, and early stopping. The regularization-based variants are discussed in \ref{app:nn_derivation}. We describe dropout and early stopping methods below. 

\begin{itemize}
    \item Dropout is applied independently to the two hidden layers during training. Let \(\h m_i^{(\ell)}\in\{0,1\}^{m_\ell}\) denote the dropout mask for layer \(\ell\in\{1,2\}\), whose components satisfy
\(
m_{ik}^{(\ell)}\sim\operatorname{Bernoulli}(1-p_{\rm drop}),
\;
k=1,\ldots,m_\ell.
\)
The hidden-layer representations become 
\[
\widetilde{\h h}_i^{(1)}
=
\frac{\h m_i^{(1)}}{1-p_{\rm drop}}
\odot
\phi_\beta(V_1\h a_i+\h s_1),
\qquad
\widetilde{\h h}_i^{(2)}
=
\frac{\h m_i^{(2)}}{1-p_{\rm drop}}
\odot
\phi_\beta
\left(
V_2\widetilde{\h h}_i^{(1)}+\h s_2
\right).
\]
The corresponding class-one probability is
\(
p_i(\h x;\h m_i)
=
\sigma
\left(
\h v_3^{\mathsf T}\widetilde{\h h}_i^{(2)}+s_3
\right).
\)

\item Early stopping does not modify the training objective. Instead, a portion of the training data is held out as a validation set, and the validation loss is monitored during training. Training is terminated when the validation loss fails to decrease by at least a prescribed minimum improvement for a specified number of consecutive epochs, referred to as the patience. The validation split, patience, and minimum improvement values are specified in Appendix~\ref{app:nn-tuning}.
\end{itemize}

\subsubsection{Hyperparameter Selection}
\label{app:nn-tuning}

For each proposed network variant, the parameter
\(\gamma\) is selected separately within each replication using \(10\)-fold cross validation over the grid
\[
\gamma
\in
\{
0.005,0.01,0.02,0.05,0.10,0.20,0.30,0.50,
0.70,1.00,1.50,2.00
\}.
\]
For each candidate \(\gamma\), the model is fitted on nine folds and evaluated on the remaining fold. We compute the median validation accuracy across the ten folds for each candidate value of \(\gamma\), and select the value with the highest median accuracy. We then refit the proposed model on the full training set using the selected value of \(\gamma\).

We fix the remaining hyperparameters across replications. For the proposed
model, we set
\(
c_{\h x}=10^{-5},
\ \text{and }
c_{\h w}=0.5.
\)
For the \(\ell_2\)-regularized variant, \(\lambda=10^{-4}\). The dropout probability $p_{\rm drop}$ is \(0.20\), while early stopping uses a \(20\%\) validation split, a patience of \(30\) epochs, and a minimum improvement of \(10^{-4}\).

\subsubsection{Implementation}
\label{app:nn-implementation}

Both proposed and baseline models use a fully connected neural network. In our experiments, we use \(m_1=m_2=256\). The two hidden layers use the Softplus activation with \(\beta=5\), and the output layer uses the logistic function. Network parameters \(V_1\), \(V_2\), and \(\h v_3\) are initialized identically across the proposed and baseline models within each replication using He-normal initialization \cite{he2015delving}. The intercept parameters \(\h s_1\), \(\h s_2\), and \(s_3\) are initialized to zero.

The \(\h x\)-subproblem in Algorithm~\ref{alg:awcs-nn} is solved using Adam with learning rate \(10^{-3}\), mini-batch size \(64\), and at most \(1{,}000\) epochs. For the proposed model, the observation weights are initialized at \(\h w^0=\mathbf{1}_n\).
The wall-clock training times at \(\alpha=0.35\) are reported in Table~\ref{tab:nn_runtime_alpha035}.

\begin{table}[t!]
\centering
\caption{Wall-clock training time (seconds) for the neural network experiment
at \(\alpha=0.35\). Entries report the mean and standard deviation over
200 replications.}
\label{tab:nn_runtime_alpha035}
\scriptsize
\setlength{\tabcolsep}{8pt}
\renewcommand{\arraystretch}{1.12}
\begin{tabular}{lcc}
\toprule
Regularization & Baseline & SWAPE \\
\midrule
None & 20.82 (0.21) & 285.12 (82.02) \\
Early stopping & 1.04 (0.46) & 1.13 (0.52) \\
$\ell_2$ & 20.79 (0.25) & 277.23 (81.64) \\
Dropout & 22.97 (1.56) & 391.26 (129.18) \\
\bottomrule
\end{tabular}
\end{table}

\section{Recommender System} \label{app:rec}
We report the details of the recommender system experiment shown in Section~\ref{sec:rec_experiments}. Section~\ref{app:rec-data} discusses the data pre-processing we applied to the Netflix Prize data and the random noise generation. 
Our algorithm and subproblem details are in Section~\ref{app:rec-sub}, while benchmark methods, hyperparameter selection process, and other implementation details are shown in Section~\ref{app:rec-comp-details}. 

\subsection{Data Processing} \label{app:rec-data}
\subsubsection{Pre-processing Netflix Prize Data} \label{app:rec-datagen_netflix}
The Netflix dataset contains over 100 million ratings, from $480,189$ users on $17,770$ movies, where each user gives a rating between $1$ and $5$ for a movie they watched. To obtain a subset, we first removed movies with fewer than \(192\) ratings and users with fewer than \(39\) ratings. The pruning process eliminates users and movies with few ratings so that each retained user and movie has enough observed ratings to support train--test splitting and model evaluation. The remaining movies were partitioned into low-, medium-, and high-popularity strata containing \(4{,}341\), \(4{,}395\), and \(4{,}406\) movies, respectively. Likewise, users were divided into low-, medium-, and high-activity strata containing \(53{,}744\), \(10{,}550\), and \(182\) users. To preserve the original data structure while shrinking the data, we sampled \(2\%\), \(5\%\), and \(10\%\) of each stratum, respectively. Intersecting the sampled sets of users and movies yields a final matrix containing \(1{,}622\) users, \(747\) movies, and \(136{,}579\) ratings, with a density of \(11.27\%\).

\subsubsection{Generating Random Noise} \label{app:rec-datagen_contamination}

Let \( A^{\mathrm{tr}}\) denote the training matrix. For each observed entry \( A^{\mathrm{tr}}_{ij}\neq 0\), we either keep the value with probability $(1-\alpha)$ or assign a random rating from $\{1, \dots, 5\}$ such that the new value differs from the original value with probability $\alpha$. We use $\alpha = 0.1$ and $\alpha = 0.3$ in our experiment.

\subsection{Details of the Algorithm}
\label{app:rec-sub}

This section details the algorithm for solving the columnwise problem \eqref{eq:slim-colj-weighted}. For simplicity of notation, we use $\x$, $\w$, and $\h a$ to represent columns of $X \in \mathbb{R}^{d \times d}$, $W \in \mathbb{R}^{n \times d}$, and $A \in \mathbb{R}^{n \times d}$ matrices, without the column index (a superscript $j$ used in the manuscript). 
Applying $R(\x) = \| \x \|_1$ and $\rho^{\rm mcp}$ defined in \eqref{eq:rho-mcp-reg}, the problem \eqref{eq:slim-colj-weighted} is equivalent to
\begin{equation}
\label{eq:slim-colj-weighted2}
\begin{aligned}
\minimize \limits_{\h x \in \mathbb{R}^d_{+},\,\h w \in \mathbb{R}^n_{+}} 
F^{\rm RS}(\x, \w) =
\tfrac12\,
  \| \h w^{1/2} 
  \odot \bigl(\h a - A\h x \bigr)\|_{2}^{2} + 
  \lambda\, \| \x \|_1+ 
  \gamma\, \sum \limits_{i=1}^n \rho^{\rm mcp}(w_i) 
\text{ subject to }
x_{j}=0.
\end{aligned}
\end{equation}
Algorithm~\ref{alg:awcs-inline} presents our method for solving the columnwise problem \eqref{eq:slim-colj-weighted2}. Note that the $\w$-subproblem \eqref{eq:awcs-rec-wsub} has an identical structure to one used in the linear regression experiment, shown in \eqref{eq:awcs-lasso-sum-wsub}, hence omitted here. 

\begin{algorithm}[H]
\small
\caption{Proposed Framework for Solving \eqref{eq:slim-colj-weighted2}}
\label{alg:awcs-inline}
\begin{algorithmic}[1]
  \REQUIRE \(A\in\mathbb{R}^{n\times d},\,\lambda,\gamma>0,\,
  c_\h x,c_\h w>0,\,\varepsilon>0,\,T_{\max}\)
\STATE Initialize
    \(\h x^{0}\geq0\) with \(x_j=0\),
    and \(\h w^{0}\geq0\)
  \STATE
\(F^{{\rm RS},0}
=
F^{\rm RS}\bigl(\h x^{0},\h w^{0}\bigr)\),
\(t\gets0\)

  \WHILE{\(\tfrac{\lvert F^{{\rm RS},t}-F^{{\rm RS},t-1}\rvert}{\lvert F^{{\rm RS},t}\rvert}>\varepsilon\) and \(t<T_{\max}\)}
    \vspace{3pt}
    \STATE Solve the following subproblems:
\begin{subequations}
\label{eq:awcs-rec-subprobs}
\begin{align}
\h x^{t+1}
&=
\argmin_{\h x\in\mathbb{R}^{d}_{+};x_{j}=0}
\Biggl[
\left(
\h w^{t}
\odot
(A\h x^{t}-\h a)
\right)^{\mathsf T}
A
\left(\h x-\h x^{t}\right)
+
\lambda\|\h x\|_1
+
\frac{c_\h x}{2}
\left\|
\h x-\h x^{t}
\right\|_2^2
\Biggr],
\label{eq:awcs-rec-xsub}
\\[6pt]
\h w^{t+1}
&=
\argmin_{\h w\in\mathbb{R}_+^{n}}
\Biggl[
\frac12
\left\|
(\h w)^{1/2}
\odot
(\h a-A\h x^{t+1})
\right\|_2^2
+
\gamma
\sum_{i=1}^{n}
\rho^{\rm mcp}(w_i)
+
\frac{c_\h w}{2}
\left\|
\h w-\h w^{t}
\right\|_2^2
\Biggr].
\label{eq:awcs-rec-wsub}
\end{align}
\end{subequations}
    \STATE
\(F^{{\rm RS},t+1}
\gets
F^{\rm RS}
\bigl(\h x^{t+1},\h w^{t+1}\bigr)\)
    \STATE \(t\gets t+1\)

  \ENDWHILE
  \STATE \textbf{return} \((\h x^{t+1},\h w^{t+1})\)
\end{algorithmic}
\end{algorithm}
\subsubsection{Closed-Form Solution for $\x$-subproblem \eqref{eq:awcs-rec-xsub}}
Next, we derive a closed-form solution for the $\x$-subproblem. Defining
$\h z^{t}
=
\h x^{t}
-
\frac{1}{c_\h x}
A^{\mathsf T}
\left(
\h w^{t}
\odot
(A\h x^{t}-\h a)
\right)
$,
the $j$-th column problem in \eqref{eq:awcs-rec-xsub} is equivalent to
\begin{align}
\h x^{t+1}
=
\argmin_{\h x\in\mathbb{R}^{d}_{+}; \, x_{j}=0}
\frac{c_\h x}{2}
\bigl\|\h x-\h z^{t}\bigr\|_2^2
+
\lambda \|\h x\|_1.
\label{eq:prox-reduced-appendix}
\end{align}
Because \(\h x\ge 0\), the \(\ell_1\)-norm reduces to \(\mathbf{1}_d^\mathsf{T}\h x\), and \eqref{eq:prox-reduced-appendix}
separates componentwise. For each \(i\neq j\), where $i \in \{1, \dots, d\}$, we solve
\begin{align}
\min_{x_i\ge 0}
\frac{c_\h x}{2}
\bigl(x_i-z_i^{t}\bigr)^2
+\lambda x_i,
\label{eq:scalar-x-appendix}
\end{align}
whose unconstrained first-order condition gives
$
c_\h x
\bigl(x_i-z_i^{t}\bigr)
+\lambda=0.
$
Hence,
$
x_i
=
z_i^{t}
-\frac{\lambda}{c_\h x},
$
if this value is nonnegative, it is feasible and, by convexity of the subproblem, is the minimizer. Otherwise, the unconstrained minimizer lies below the feasible set \(x_i\ge0\). In this case, the derivative of the scalar objective at \(x_i=0\) is nonnegative, since the derivative of a convex function is nondecreasing, the objective is nondecreasing for all \(x_i\ge0\). Therefore, the minimum is attained at the boundary \(x_i=0\). Consequently,
\[
x_i^{t+1}
=
\max\!\left\{
z_i^{t}-\frac{\lambda}{c_\h x},
\,0
\right\},
\]
for all $i\neq j$ where $j$ is the index of the column
that minimizes the current objective function. Otherwise, $x^{t+1}_{j} = 0$.

\subsection{Computational Details} \label{app:rec-comp-details}

\subsubsection{Comparison Methods}
\label{app:rec-comparison}
We present five benchmark methods: singular value decomposition (SVD), User-based \(k\)-nearest neighbors (UserKNN), Item-based \(k\)-nearest neighbors (ItemKNN), SLIM, and SLIM-LTS, which is a method that applies the LTS in Section~\ref{app:reg-comparison} to the users of the data matrix.

\paragraph{Sparse Linear Methods for Recommendation (SLIM)}
Given the user--item interaction matrix \(A\in\mathbb{R}^{n\times d}\), the sparse linear method (SLIM) \cite{6137254} learns an item--item similarity matrix \(X\in\mathbb{R}^{d\times d}\) by solving\begin{align}
\min_{X\in\mathbb{R}^{d\times d}}
\quad &
\frac{1}{2}\|A-AX\|_F^2
+\frac{\beta}{2}\|X\|_F^2
+\lambda\|X\|_1
\label{eq:slim-matrix-app}\\
\text{s.t.}\quad &
\operatorname{diag}(X)=0,\qquad X\ge 0.
\nonumber
\end{align}
The zero-diagonal constraint prevents trivial self-reconstruction, while the nonnegativity constraint promotes additive item--item relations. Since the objective separates across columns, \eqref{eq:slim-matrix-app} decomposes into \(d\) independent subproblems. For each item \(j\in[d]\), the \(j\)-th column \(x\in\mathbb{R}^d\) is obtained from
\begin{align}
\min_{\h x\in\mathbb{R}^{d}}
\quad &
\frac{1}{2}\|\h a-A\h x\|_2^2
+\frac{\beta}{2}\|\h x\|_2^2
+\lambda\|\h x\|_1
\label{eq:slim-column-app}\\
\text{s.t.}\quad &
X_{jj}=0,\qquad x\ge0.
\nonumber
\end{align}
\paragraph{Robust SLIM via Least Trimmed Squares (SLIM--LTS)} 
To improve robustness against corrupted or adversarial ratings, we also consider an LTS-based variant of SLIM \eqref{eq:slim-column-app}. Let
$
\Omega_j:=\{\,i\in\{1,\ldots,n\}\mid A_{ij}\neq 0\,\}
$
denote the set of observed entries in column \(j\), and define the residual
$
r_{i}(\x^j):=(\h a^j -A\h x^j)_i.
$
Let \(h\) denote the trimming parameter, and let \(\mathcal{I}_j\subseteq \Omega_j\) denote the index set corresponding to the smallest fraction \(h\) of squared residuals. Then SLIM--LTS solves
\begin{align}
\min_{\x^j}
\quad &
\sum_{i\in \mathcal{I}_j} r_{i}^2(\h x^j)
+\frac{\beta}{2}\|\h x^j \|_2^2
+\lambda\|\h x^j \|_1
\label{eq:slim-lts-app}\\
\text{s.t.}\quad &
\mathcal{I}_j\subseteq \Omega_j,\quad |\mathcal{I}_j| = \lfloor h|\Omega_j| \rfloor,\quad
X_{jj}=0,\quad \x^j \ge0.
\nonumber
\end{align}
Only the \(h\) smallest squared residuals are retained in the objective, thereby limiting the influence of extreme outliers. The optimization is performed by a FAST--LTS style alternating scheme \cite{rousseeuw2006computing}.
\paragraph{Matrix Factorization (SVD)}
We also compare against a classical latent-factor model based on singular value decomposition \cite{koren2009matrix}. This approach learns user and item latent factors by solving
\begin{align}
\min_{\{P^i\in\mathbb{R}^{1\times k},\,Q^j\in\mathbb{R}^{k}\}}
\quad &
\sum_{i=1}^{n}\sum_{j=1}^{d}\bigl(A_{ij}-\langle P^i,Q^j\rangle\bigr)^2
+
\lambda\left(
\sum_{i=1}^{n}\|P^i\|_2^2
+
\sum_{j=1}^{d}\|Q^j\|_2^2
\right).
\label{eq:svd-app}
\end{align}
Here, \(P\in\mathbb{R}^{n\times k}\) and \(Q\in\mathbb{R}^{k\times d}\) denote the user and item latent factor matrices, respectively, and \(k\) is the latent dimension. The predicted preference score for user \(i\) and item \(j\) is given by
$
\hat A_{ij}=P^iQ^j,
$ and these scores are used to rank unseen items.

\paragraph{User \(k\)-Nearest Neighbors (UserKNN) and Item \(k\)-Nearest Neighbors (ItemKNN)} Finally, we consider both user-based and item-based \(K\)-nearest neighbor methods. These methods predict preferences using local similarity structure in the user--item interaction matrix \cite{sarwar2001item}.

In UserKNN, similarity between users \(i\) and \(u\) is measured by the Pearson correlation computed over their co-rated items:
\[
\operatorname{sim}(i,u)
=
\frac{
\sum_{j\in \Omega_i\cap \Omega_u}(A_{ij}-\bar A^i)(A_{uj}-\bar A^u)
}{
\sqrt{\sum_{j\in \Omega_i\cap \Omega_u}(A_{ij}-\bar A^i)^2}
\sqrt{\sum_{j\in \Omega_i\cap \Omega_u}(A_{uj}-\bar A^u)^2}
}.
\]
where $\Omega_i$ denotes the set of items rated by user $i$, and $\bar A^i$ is
the mean rating of user $i$ over $\Omega_i$.
Let \(\mathcal N_i^{(K)}\) denote the set of the \(K\) most similar users to \(i\). The predicted rating for user \(i\) on item \(j\) is
\[
\hat A_{ij}
=
\bar A^i
+
\frac{
\sum_{u\in\mathcal N_i^{(K)}} \operatorname{sim}(i,u)(A_{uj}-\bar A^u)
}{
\sum_{u\in\mathcal N_i^{(K)}} |\operatorname{sim}(i,u)|
}.
\]
ItemKNN is defined analogously, except that similarities are computed between items rather than users. Let \(\mathcal N_j^{(K)}\) denote the set of the \(K\) nearest neighbors of item \(j\). Then the predicted score for user \(i\) on item \(j\) is
\[
\hat A_{ij}
=
\bar a
+
\frac{
\sum_{\ell\in\mathcal N_j^{(K)}} \operatorname{sim}(j,\ell)(A_{i\ell}-\bar A^\ell)
}{
\sum_{\ell\in\mathcal N_j^{(K)}} |\operatorname{sim}(j,\ell)|
}.
\]
where $\bar a$ is the mean rating of item $j$.

\subsubsection{Hyperparameter Selection}
\label{app:rec-tuning}

For model selection, we used one randomly selected replication and applied five-fold cross-validation to its training portion using the same per-user item-holdout strategy.
For Algorithm~\ref{alg:awcs-inline}, we set the parameters as follows: 
$\gamma = 0.1$, $\eta = 2.0$, $c_{\h x} =600,000$, $c_{\h w} = 0.1$, $\varepsilon = 10^{-4}$, and $T_{\max} = 100$.  The remaining hyperparameters $\lambda$ and $\theta$ are chosen by cross-validation. The candidates are
\[
  \lambda \in \{0.001, 0.005, 0.01,\, 0.05,\,0.1,\,0.5,\,1\}, \quad
  \theta \in \{0.001, 0.005, 0.01,\,0.05,\,0.1,\,0.5,\,1\},
\]
yielding \(49\) parameter combinations. For SLIM, 
we tune the regularization parameters $(\lambda, \beta)$ over the grids
\[
\lambda \in \{0.001, 0.005, 0.01,\, 0.05,\, 0.1,\, 0.5,\, 1\}, \quad
\beta \in \{0.001, 0.005, 0.01,\, 0.05,\, 0.1,\, 0.5,\, 1\}.
\]
For SLIM-LTS, we use the same $(\lambda, \beta)$ values selected for SLIM, while fixing the trimming parameter $h=0.8$. For SVD, we tune the regularization parameter over the grid
$
\lambda \in \{0.0001, 0.001, 0.01, 0.1, 1.0\}.
$
For neighborhood-based methods (UserKNN and ItemKNN), we tune the number of neighbors as
$K \in \{10, 20, 30, 40, 60, 80\}.$
For each parameter pair, the mean and standard deviation of
HitRate@5 across folds were recorded, and the pair with the highest mean was selected for final training on the complete dataset. This process selected $\lambda = 0.001$ for SVD. For both UserKNN and ItemKNN, $K=10$ was selected.  Table~\ref{tab:hyper_netflix} summarizes the final hyperparameter settings used for the remaining method. Bold entries correspond to values selected through cross-validation, whereas all other parameters were fixed across experiments.
\begingroup
    \renewcommand{\bfseries}{\relax} 
\begin{table}[H]
\scriptsize
\centering
\caption{Hyperparameter settings for each recommendation method}
\label{tab:hyper_netflix}
\setlength{\tabcolsep}{6pt}
\begin{tabular}{l c c c c c c c c c}
\toprule
Method & $\lambda$ & $\theta$ & $\gamma$ & $\eta$ & $c_{\h{x}}$ & $c_{\h{w}}$ & $\beta$ & $h$ & $K$ \\
\midrule
SWAPE      & \textbf{0.05} & \textbf{0.5} & 0.1 & 2.0 & 600000 & 0.1 & -- & -- & -- \\
SLIM      & \textbf{0.5} & -- & -- & -- & -- & --& \textbf{1.0}  & -- & -- \\
SLIM-LTS & 0.5 & -- & -- & -- & -- & --& 1.0 & 0.8 & -- \\

\bottomrule
\end{tabular}
\end{table}
\endgroup

\subsubsection{Implementation}

SLIM and SLIM--LTS were solved using Gurobi Optimizer 12.0 through its Python interface \cite{gurobi}.
SVD, UserKNN, and ItemKNN were implemented using the \texttt{Surprise} library \cite{hug2020surprise}.

All runtimes are reported as wall-clock time and averaged over 30 replications under the first setting discussed in the Results section. For fairness, SLIM, SLIM--LTS, and the proposed method were executed using 128 CPU cores, while SVD and KNN relied on their internal optimized implementations. The proposed method leverages column-wise parallelization across cores. Table~\ref{tab:rec-runtime} reports the corresponding wall-clock training times. UserKNN and ItemKNN require substantially more time, due to the cost of computing and applying similarity measures across users or items. Runtime for the SLIM-based methods depends on the solver and implementation.

\begin{table}[H]
\centering
\scriptsize
\caption{Training wall-clock time (seconds) for the first
setting. Reported values are the mean
(standard deviation) over 30 independent replications.}
\label{tab:rec-runtime}
\begin{tabular}{lc}
\toprule
Method & Training Time \\
\midrule
SWAPE  & 1.47 (0.07) \\
SLIM      & 1512.63 (45.27) \\
SLIM--LTS & 2277.18 (67.54) \\
SVD       & 0.66 (0.02) \\
UserKNN   & 44.37 (0.88) \\
ItemKNN   & 29.23 (0.57) \\
\bottomrule
\end{tabular}
\end{table}

\subsection{Additional Results and Analysis}
\label{app:additional-recsys-results}
Figure~\ref{fig:recsys_prediction_scores} provides additional insight into the predicted scores from the first recommender system experiment. These results illustrate that strong ranking performance does not necessarily imply close agreement between predicted-score magnitudes and observed ratings.
\begin{figure}[H]
    \centering    \includegraphics[width=0.70\textwidth]{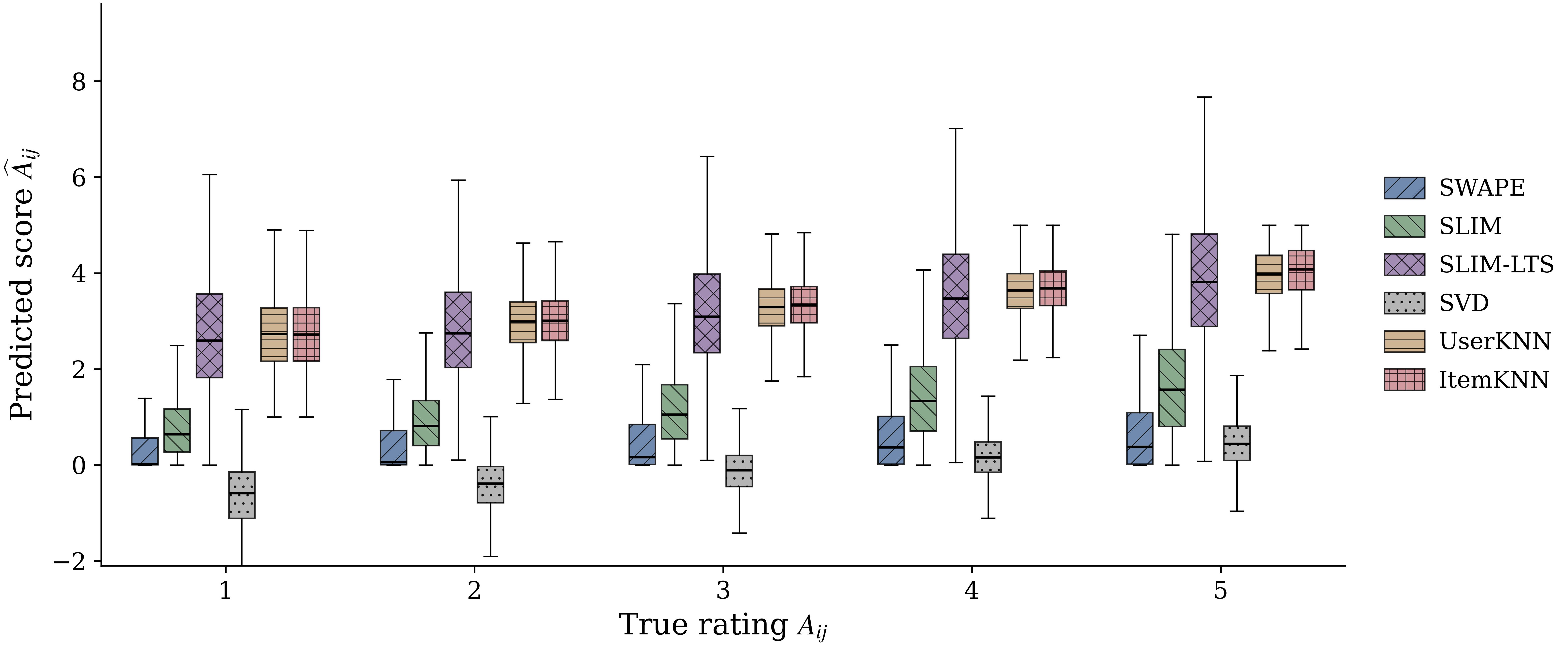}
    \caption{Distribution of the predicted scores $\widehat{A}_{ij}$ for a replication of the first recommender system experiment. SLIM generally produces scores below the corresponding true ratings $A_{ij}$, which is consistent with the sparsity-inducing regularization in its formulation that shrinks the estimated coefficients toward zero. SVD exhibits a similarly compressed score distribution. These results show that methods achieving relatively strong hit rate performance may still produce predicted scores whose magnitudes differ substantially from the observed ratings. For each true rating $A_{ij}\in\{1,\ldots,5\}$, each box summarizes the predicted scores for all observed test user-item pairs having that rating.}    \label{fig:recsys_prediction_scores}
\end{figure}

\end{document}